\documentclass{tac}

\usepackage{amsmath}
\usepackage{amssymb}
\usepackage{enumitem}

\usepackage[all,cmtip]{xy}

\input diagxy

\usepackage[colorlinks=true]{hyperref}
\hypersetup{allcolors=[rgb]{0.1,0.1,0.4}}

\author{Hongliang Lai, Lili Shen and Walter Tholen}

\thanks{Partial financial assistance by National Natural Science Foundation of China (11101297),
International Visiting Program for Excellent Young Scholars of Sichuan University
and the Natural Sciences and Engineering Research Council (NSERC) of Canada is gratefully acknowledged.
This work was completed while the first author held a Visiting Professorship at York University.}

\address{School of Mathematics, Sichuan University\\
 Chengdu 610064, China\\[5pt]
 Department of Mathematics and Statistics, York University\\
 Toronto, Ontario M3J 1P3, Canada\\
}

\title{Lax distributive laws for topology, II}

\copyrightyear{2016}


\eaddress{hllai@scu.edu.cn \CR shenlili@scu.edu.cn \CR tholen@mathstat.yorku.ca}

\newtheorem{thm}{Theorem}

\newtheorem{lem}{Lemma}
\newtheorem{prop}{Proposition}
\newtheorem{cor}{Corollary}

\newtheoremrm{rem}{Remark}
\newtheoremrm{defn}{Definition}
\newtheoremrm{exmp}{Example}

\mathrmdef{Hom}
\mathbfdef{Set}

\DeclareMathOperator{\ob}{ob}

\def\oto{{\bfig\morphism<180,0>[\mkern-4mu`\mkern-4mu;]\place(86,0)[\circ]\efig}}
\def\rto{{\bfig\morphism<180,0>[\mkern-4mu`\mkern-4mu;]\place(78,0)[\mapstochar]\efig}}

\newcommand{\da}{\downarrow}
\newcommand{\ua}{\uparrow}

\newcommand{\lda}{\swarrow}
\newcommand{\rda}{\searrow}

\newcommand{\Lra}{\Longrightarrow}

\newcommand{\bv}{\bigvee}
\newcommand{\bw}{\bigwedge}

\newcommand{\dv}{\dashv}

\newcommand{\od}{\odot}
\newcommand{\opl}{\oplus}
\newcommand{\ola}{\overleftarrow}
\newcommand{\ora}{\overrightarrow}

\newcommand{\ie}{\text{\rm !`}}

\renewcommand{\phi}{\varphi}
\newcommand{\al}{\alpha}
\newcommand{\be}{\beta}

\newcommand{\ep}{\varepsilon}

\newcommand{\Lam}{\Lambda}
\newcommand{\lam}{\lambda}

\newcommand{\si}{\sigma}

\newcommand{\CC}{\mathcal{C}}

\newcommand{\CF}{\mathcal{F}}

\newcommand{\CQ}{\mathcal{Q}}
\newcommand{\calR}{\mathcal{R}}

\newcommand{\sD}{{\sf D}}

\newcommand{\sP}{{\sf P}}

\newcommand{\sV}{{\sf V}}

\newcommand{\sk}{{\sf k}}

\newcommand{\sfs}{{\sf s}}
\newcommand{\sy}{{\sf y}}

\newcommand{\bbI}{\mathbb{I}}
\newcommand{\bbP}{\mathbb{P}}

\newcommand{\bbT}{\mathbb{T}}

\newcommand{\Fs}{\mathfrak{s}}

\newcommand{\Fy}{\mathfrak{y}}

\newcommand{\Alg}{{\bf Alg}}

\newcommand{\Cat}{{\bf Cat}}

\newcommand{\Cls}{{\bf Cls}}

\newcommand{\Dist}{{\bf Dist}}
\newcommand{\Inf}{{\bf Inf}}
\newcommand{\Int}{{\bf Int}}

\newcommand{\Mon}{{\bf Mon}}

\newcommand{\Rel}{{\bf Rel}}
\renewcommand{\Set}{{\bf Set}}

\newcommand{\Sup}{{\bf Sup}}
\newcommand{\QCat}{\CQ\text{-}\Cat}

\newcommand{\QCls}{\CQ\text{-}\Cls}

\newcommand{\QDist}{\CQ\text{-}\Dist}

\newcommand{\QInf}{\CQ\text{-}\Inf}
\newcommand{\QInt}{\CQ\text{-}\Int}

\newcommand{\QRel}{\CQ\text{-}\Rel}
\newcommand{\QSup}{\CQ\text{-}\Sup}
\newcommand{\VCat}{\sV\text{-}\Cat}

\newcommand{\VRel}{\sV\text{-}\Rel}

\newcommand{\dphi}{\phi^{\da}}
\newcommand{\uphi}{\phi_{\ua}}

\newcommand{\dpsi}{\psi^{\da}}

\newcommand{\sPd}{\sP^{\dag}}
\newcommand{\syd}{\sy^{\dag}}

\newcommand{\co}{{\rm co}}
\newcommand{\op}{{\rm op}}

\newcommand{\PX}{\sP X}
\newcommand{\PY}{\sP Y}

\newcommand{\PdX}{\sPd X}
\newcommand{\PdY}{\sPd Y}

\newcommand{\DV}{\sD\sV}

\newcommand{\hT}{\hat{T}}
\newcommand{\hP}{\hat{\sP}}

\newcommand{\TQ}{(\bbT,\CQ)}
\newcommand{\TTQ}{(\bbT,\hT,\CQ)}

\newcommand{\TV}{(\bbT,\sV)}
\newcommand{\TQCat}{\TQ\text{-}\Cat}
\newcommand{\TTQCat}{\TTQ\text{-}\Cat}
\newcommand{\TVCat}{\TV\text{-}\Cat}

\newcommand{\ophi}{\ora{\phi}}
\newcommand{\opsi}{\ora{\psi}}

\newcommand{\olal}{\ola{\al}}
\newcommand{\olbe}{\ola{\be}}
\newcommand{\olphi}{\ola{\phi}}
\newcommand{\olpsi}{\ola{\psi}}

\newcommand{\rc}{\rm c}

\newcommand{\lamd}{\lam^{\dag}}
\newcommand{\Lamd}{\Lam^{\dag}}
\newcommand{\ssd}{\sfs^{\dag}}
\newcommand{\Fsd}{\Fs^{\dag}}
\newcommand{\Fyd}{\Fy^{\dag}}
\newcommand{\PP}{\sP\sP}
\newcommand{\PdP}{\sPd\sP}
\newcommand{\PPd}{\sP\sPd}

\newcommand{\PPX}{\sP\PX}
\newcommand{\PdPX}{\sPd\PX}
\newcommand{\PPdX}{\sP\PdX}
\newcommand{\PdPdX}{\sPd\PdX}

\newcommand{\PdPY}{\sPd\PY}
\newcommand{\PPdY}{\sP\PdY}

\newcommand{\lamQAlg}{(\lam,\CQ)\text{-}\Alg}
\newcommand{\lamdQAlg}{(\lamd,\CQ)\text{-}\Alg}
\newcommand{\LamQAlg}{(\Lam,\CQ)\text{-}\Alg}
\newcommand{\LamdQAlg}{(\Lamd,\CQ)\text{-}\Alg}
\renewcommand{\leq}{\leqslant}
\renewcommand{\geq}{\geqslant}

\newcommand{\dal}{\al^{\da}}
\newcommand{\dbe}{\be^{\da}}

\numberwithin{equation}{section}

\begin{document}

\maketitle
\begin{abstract}
For a small quantaloid $\mathcal{Q}$ we consider four fundamental 2-monads $\mathbb{T}$ on $\mathcal{Q}\text{-}{\bf Cat}$, given by the presheaf 2-monad $\mathbb{P}$ and the copresheaf 2-monad $\mathbb{P}^{\dag}$, as well as by their two composite 2-monads, and establish that they all laxly distribute over $\mathbb{P}$. These four 2-monads therefore admit lax extensions to the category $\mathcal{Q}\text{-}{\bf Dist}$ of $\mathcal{Q}$-categories and their distributors. We characterize the corresponding $(\mathbb{T},\mathcal{Q})$-categories in each of the four cases, leading us to both known and novel categorical structures.
\end{abstract}

\section{Introduction}

 \emph{Monoidal Topology} \cite{Hofmann2014} provides a common framework for the study of fundamental metric and topological structures. Its ingredients are a quantale $\sV$, a $\Set$-monad $\bbT$ and, most importantly, a lax extension of $\bbT$ to the 2-category $\VRel$ of sets and $\sV$-valued relations. Such lax extensions are equivalently described by lax distributive laws of $\bbT$ over the \emph{discrete} $\sV$-presheaf monad
$\bbP_{\sV}$, the Kleisli category of which is exactly $\VRel$. Once equipped with a lax extension or lax distributive law, the monad $\bbT$ may then be naturally extended to become a 2-monad on the 2-category $\VCat$. This lax monad extension from $\Set$ to $\VCat$ facilitates the study of greatly enriched structures. For example, for $\sV$ the two-element chain and $\bbT$ the ultrafilter monad, while the Eilenberg-Moore category over $\Set$ is ${\bf CompHaus}$, over $\VCat$ one obtains \emph{ordered }compact Hausdorff spaces, and when $\sV$ is Lawvere's \cite{Lawvere1973} extended half-line $[0,\infty]$, \emph{metric} compact Hausdorff spaces; see \cite{Nachbin1948,Tholen2009,Hofmann2014}. Moreover, the functorial interaction between the Eilenberg-Moore category $(\VCat)^{\bbT}$ and the category $\TVCat$ of $\TV$-categories is a pivotal step for a serious study of \emph{representability}, a powerful property which, in the basic example of the two-element chain and the ultrafilter monad, entails core-compactness, or exponentiability, of topological spaces; see \cite{Clementino2009} and \cite[Section III.5]{Hofmann2014}.

While this mechanism for generating a 2-monad on $\VCat$ from a $\Set$-monad provides an indispensable tool in monoidal topology, the question arises whether it is possible to make a given 2-monad $\bbT$ on $\VCat$ the starting point of a satisfactory theory, preferably even in the more general context of a small quantaloid $\CQ$, (\emph{i.e.}, a $\Sup$-enriched category), rather than just a quantale $\sV$ (\emph{i.e.}, a $\Sup$-enriched monoid), a context that has been propagated in this paper's predecessor \cite{Tholen2016}. Such theory should, as a first step, entail the study of lax extensions of $\bbT$ to the 2-category $\QDist$ of $\CQ$-categories and their \emph{distributors} (also \emph{(bi)modules}, or \emph{profunctors}), rather than just to $\QRel$, or, equivalently, the study of lax distributive laws of $\bbT$ over the \emph{non-discrete} presheaf monad $\bbP_{\CQ}$, rather than over its discrete counterpart. The fact that the non-discrete presheaf monad is, other than its discrete version, lax idempotent (\emph{i.e.}, of Kock-Z\"oberlein type \cite{Zoeberlein1976,Kock1995}), serves as a first indicator that this approach should in fact lead to a categorically more satisfactory theory.

This paper makes the case for an affirmative answer to the question raised, even in the extended context of a given small quantaloid $\CQ$, rather than that of a quantale. It is centred around a fairly simple, but fundamental extension procedure for endo-2-functors of $\CQ$-{\bf Cat} to become lax endofunctors of $\QDist$, which has been used in the quantalic context in \cite{Akhvlediani2010} and extended from quantales to quantaloids in \cite{Lai2016a}. More importantly, the paper \cite{Lai2016a} emphasized the fact that there is precisely one \emph{flat} (or \emph{normal}) lax extension when the given endo-2-functor of $\QCat$ preserves the full fidelity of $\CQ$-functors. We recall this technique in Section \ref{Flat_Dist_Laws} and then apply it to
four naturally arising 2-monads $\bbT$ on $\QCat$ which do not come about as monads ``lifted'' from {\bf Set} via the mechanism described above, but which should nevertheless be of considerable general interest. They all distribute laxly, but flatly, over $\bbP=\bbP_{\CQ}$ and, hence, are laxly, but flatly, extendable to $\QDist$, and we give a detailed description of the respective lax algebras, or $\TQ$-categories, arising. These monads are
\begin{enumerate}[label=$\bullet$]
\item the presheaf 2-monad $\bbP$ itself (Section \ref{Law_P});
\item the copresheaf 2-monad $\bbP^{\dag}$ (Section \ref{Law_P});
\item the double presheaf 2-monad $\bbP\bbP^{\dag}$ (Section \ref{Law_PPd});
\item the double copresheaf 2-monad $\bbP^{\dag}\bbP$ (Section \ref{Law_PdP}).
\end{enumerate}
In each of the four cases, the establishment of the needed lax distributive law over $\bbP$ and the characterization of the corresponding lax algebras, or, equivalently, $\TQ$-categories, take some ``technical'' effort. However, the lax algebras pertaining to both, $\bbP$ and $\bbP^{\dag}\bbP$, are fairly quickly identifiable as \emph{$\CQ$-closure spaces}, as considered in \cite{Shen2016b,Shen2013a}. More challenging is the identification of the lax algebras pertaining to $\bbP\bbP^{\dag}$, which we describe as \emph{$\CQ$-interior spaces}, a structure considered here for the first time. Also the lax algebras pertaining to $\bbP^{\dag}$ are of a novel flavour; they are monoid objects in $\QDist$. Given that their discrete cousins, \emph{i.e.}, the monoid objects in $\QRel$, are $\CQ$-categories, they surely deserve further study.

We have given sufficiently many details to make the proofs easily verifiable for the reader, also since all needed basic tools are comprehensively listed in Section \ref{QCat}.
The introduction of lax distributive laws of a 2-monad over the (non-discrete) presheaf monad and of their lax algebras (as given in Section \ref{Lax_Dist_Laws}), as well as the proof of the fact that these correspond bijectively to lax extensions of $\bbT$ to $\QDist$, with lax algebras corresponding to $\TQ$-categories (as given in Section \ref{Dist_Law_vs_Extension}), are straightforward extensions of their ``discrete'' treatment in \cite{Tholen2016}.
Nevertheless, prior reading of \cite{Tholen2016} is not required for the purpose of understanding these parts of the  paper.

\emph{Acknowledgement.} We thank the anonymous referee for his/her suggestions on the first version of this paper,  which was missing the techniques of Section 4 that in particular simplify the treatment of the four monads of our interest. These suggestions included a concrete indication on how the results of this paper may be established in the context of a locally cocomplete bicategory, rather than that of a quantaloid, on which we plan to follow up in a separate paper.

\section{Quantaloid-enriched categories and their distributors} \label{QCat}

A \emph{quantaloid} \cite{Rosenthal1996} is a category enriched in the monoidal-closed category $\Sup$ \cite{Joyal1984} of complete lattices and $\sup$-preserving maps. Explicitly, a quantaloid $\CQ$ is a 2-category with its 2-cells given by an order ``$\preceq$'', such that each hom-set $\CQ(r,s)$ is a complete lattice and the composition of morphisms from either side preserves arbitrary suprema. Hence, $\CQ$ has ``internal homs'', denoted by $\lda$ and $\rda\,$, as the right adjoints of the composition functors:
$$-\circ u\dv -\lda u:\CQ(r,t)\to\CQ(s,t)\quad\text{and}\quad v\circ -\dv v\rda -:\CQ(r,t)\to\CQ(r,s);$$
explicitly,
$$u\preceq v\rda w\iff v\circ u\preceq w\iff v\preceq w\lda u$$
for all morphisms $u:r\to s$, $v:s\to t$, $w:r\to t$ in $\CQ$.

Throughout this paper, we let $\CQ$ be a \emph{small} quantaloid. From $\CQ$ one forms a new (large) quantaloid $\QRel$ of \emph{$\CQ$-relations} with the following data: its objects are those of $\Set/\CQ_0$ (with $\CQ_0:=\ob\CQ$), \emph{i.e.}, sets $X$ equipped with an \emph{array} (or \emph{type}) map $|\text{-}|:X\to\CQ_0$, and a morphism $\phi:X\rto Y$ in $\QRel$ is a map that assigns to every pair $x\in X$, $y\in Y$ a morphism $\phi(x,y):|x|\to|y|$ in $\CQ$; its composite with $\psi:Y\rto Z$ is defined by
$$(\psi\circ\phi)(x,z)=\bv_{y\in Y}\psi(y,z)\circ\phi(x,y),$$
and $1_X^{\circ}:X\rto X$ with
$$1_X^{\circ}(x,y)=\begin{cases}
1_{|x|} & \text{if}\ x=y,\\
\bot & \text{else}
\end{cases}$$
serves as the identity morphism on $X$. As $\CQ$-relations are equipped with the pointwise order inherited from $\CQ$, internal homs in $\QRel$ are computed pointwise as
$$(\theta\lda\phi)(y,z)=\bw_{x\in X}\theta(x,z)\lda\phi(x,y)\quad\text{and}\quad(\psi\rda\theta)(x,y)=\bw_{z\in Z}\psi(y,z)\rda\theta(x,z)$$
for all $\phi:X\rto Y$, $\psi:Y\rto Z$, $\theta:X\rto Z$.

A (small) \emph{$\CQ$-category} is an (internal) monad in the 2-category $\QRel$; or equivalently, a monoid in the monoidal-closed category $(\QRel(X,X),\circ)$, for some $X$ over $\CQ_0$. Explicitly, a $\CQ$-category consists of an object $X$ in $\Set/\CQ_0$ and a $\CQ$-relation $a:X\rto X$ (its ``hom''), such that $1_X^{\circ}\preceq a$ and $a\circ a\preceq a$. For every $\CQ$-category $(X,a)$, the underlying (pre)order on $X$ is given by
$$x\leq x'\iff|x|=|x'|\ \text{and}\ 1_{|x|}\preceq a(x,x'),$$
and we write $x\cong x'$ if $x\leq x'$ and $x'\leq x$.

A map $f:(X,a)\to(Y,b)$ between $\CQ$-categories is a \emph{$\CQ$-functor} (resp. \emph{fully faithful $\CQ$-functor}) if it lives in $\Set/\CQ_0$ and satisfies $a(x,x')\preceq b(fx,fx')$ (resp. $a(x,x')=b(fx,fx')$) for all $x,x'\in X$. With the pointwise order of $\CQ$-functors inherited from $Y$, \emph{i.e.},
$$f\leq g:(X,a)\to(Y,b)\iff\forall x\in X:\ fx\leq gx\iff\forall x\in X:\ 1_{|x|}\preceq b(fx,gx),$$
$\CQ$-categories and $\CQ$-functors are organized into a 2-category $\QCat$.

The one-object quantaloids are the (unital) \emph{quantales} (see \cite{Rosenthal1990}); equivalently, a quantale is a complete lattice $\sV$ with a monoid structure whose binary operation $\otimes$ preserves suprema in each variable. We generally denote the $\otimes$-neutral element by $\sk$; so, $\sk=1_{\star}$ if we denote by $\star$ the only object of the monoid $\sV$, considered as a category.

We refer to \cite{Hofmann2014} for the standard examples of quantales $\sV$ of interest in \emph{monoidal topology}, which include the Lawvere quantale $[0,\infty]$ with its addition (\cite{Lawvere1973}) or its frame operation (\cite{Rutten1996}), giving generalized metric spaces and generalized ultrametric spaces as $\sV$-categories. For relevant examples of small quantaloids that are not quantales, we mention the fact that every quantale $\sV$ (in fact, every quantaloid) gives rise to the quantaloid $\DV$ of ``diagonals of $\sV$'' (see \cite{Stubbe2014}), which has a particularly simple description when $\sV$ is \emph{divisible}: see \cite{Hoehle2011a,Pu2012}. For $\sV=[0,\infty]$ one obtains as $\DV$-categories generalized partial (ultra-)metric spaces, as studied by various authors: \cite{Matthews1994,Bukatin2009,Hoehle2011a,Pu2012,Tao2014,Tholen2016}. We also refer to Walters' original paper \cite{Walters1982} which associates with a small site $(\CC,\CF)$ a quantaloid $\calR$
whose Cauchy complete $\calR$-categories have been identified as the internal ordered objects in the topos of sheaves over $(\CC,\CF)$ in the thesis \cite{Heymans2010}; see also \cite{Heymans2014}.

A $\CQ$-relation $\phi:X\rto Y$ becomes a \emph{$\CQ$-distributor} $\phi:(X,a)\oto(Y,b)$ if it is compatible with the $\CQ$-categorical structures $a$ and $b$; that is,
$$b\circ\phi\circ a\preceq\phi.$$
$\CQ$-categories and $\CQ$-distributors constitute a quantaloid $\QDist$ that contains $\QRel$ as a full subquantaloid, in which the composition and internal homs are calculated in the same way as those of $\CQ$-relations; the identity $\CQ$-distributor on $(X,a)$ is given by its hom $a:(X,a)\oto(X,a)$.

Each $\CQ$-functor $f:(X,a)\to(Y,b)$ induces an adjunction $f_*\dv f^*$ in $\QDist$, given by
\begin{equation} \label{graph_def}
\begin{array}{ll}
f_*:(X,a)\oto(Y,b),& f_*(x,y)=b(fx,y)\quad\text{and}\\
f^*:(Y,b)\oto(X,a),& f^*(y,x)=b(y,fx),
\end{array}
\end{equation}
and called the \emph{graph} and \emph{cograph} of $f$, respectively. Obviously, $a=(1_X)_*=1_X^*$ for any $\CQ$-category $(X,a)$; hence, $a=1_X^*$ will be our standard notation for identity morphisms in $\QDist$.


For an object $s$ in $\CQ$, and with $\{s\}$ denoting the singleton $\CQ$-category, the only object of which has array $s$ and hom $1_s$, $\CQ$-distributors of the form $\si:X\oto\{s\}$ are called \emph{presheaves} on $X$ and constitute a $\CQ$-category $\PX$, with $1_{\PX}^*(\si,\si')=\si'\lda\si$. Dually, the \emph{copresheaf} $\CQ$-category $\PdX$ consists of $\CQ$-distributors $\tau:\{s\}\oto X$ with $1_{\PdX}^*(\tau,\tau')=\tau'\rda\tau$.

It is important to note that for any $\CQ$-category $X$, it follows from the definition that the underlying order on $\PdX$ is the \emph{reverse} local order of $\QDist$, \emph{i.e.},
$$\tau\leq\tau'\ \text{in}\ \PdX\iff\tau'\preceq\tau\ \text{in}\ \QDist.$$
That is why we use a different symbol, ``$\leq$'',  for the underlying order of $\CQ$-categories and the 2-cells in $\QCat$, while ``$\preceq$'' is reserved for ordering the 2-cells in $\CQ$ and $\QDist$.

A $\CQ$-category $X$ is \emph{complete} if the \emph{Yoneda embedding}
$$\sy_X:X\to\PX,\quad x\mapsto 1_X^*(-,x),$$
has a left adjoint $\sup_X:\PX\to X$ in $\QCat$; that is,
$$1_X^*({\sup}_X\si,-)=1_{\PX}^*(\si,\sy_X-)=1_X^*\lda\si$$
for all $\si\in\PX$. It is well known that $X$ is a complete $\CQ$-category if, and only if, $X^{\op}:=(X,(1_X^*)^{\op})$ with $(1_X^*)^{\op}(x,x')=1_X^*(x',x)$ is a complete $\CQ^{\op}$-category (see \cite{Stubbe2005}), where the completeness of $X^{\op}$ may be translated as the \emph{co-Yoneda embedding}
$$\syd_X:X\to\PdX,\quad x\mapsto 1_X^*(x,-),$$
admitting a right adjoint $\inf_X:\PdX\to X$ in $\QCat$.

\begin{lem} \label{Yoneda} {\rm\cite{Shen2013a,Stubbe2005}}
Let $X$ be a $\CQ$-category.
\begin{enumerate}[label={\rm (\arabic*)}]
\item (Yoneda Lemma) For all $\si\in\PX$, $\tau\in\PdX$,
    $$\si=(\sy_X)_*(-,\si)=1_{\PX}^*(\sy_X-,\si)\quad\text{and}\quad\tau=(\syd_X)^*(\tau,-)=1_{\PdX}^*(\tau,\syd_X-).$$
    In particular, both $\sy_X:X\to\PX$ and $\syd_X:X\to\PdX$ are fully faithful.
\item $\sup_X\cdot\sy_X\cong 1_X$,\quad $\inf_X\cdot\syd_X\cong 1_X$.
\item \label{PX_sup}
    Both $\PX$ and $\PdX$ are separated\footnote{A $\CQ$-category $X$ is \emph{separated} if $x\cong x'$ implies $x=x'$ for all $x,x'\in X$.} and complete, with
    $${\sup}_{\PX}\si=\si\circ(\sy_X)_*\quad\text{and}\quad{\inf}_{\PdX}\tau=(\syd_X)^*\circ\tau,$$
    for all $\si\in\PPX$, $\tau\in\PdPdX$.
\end{enumerate}
\end{lem}

Each $\CQ$-distributor $\phi:X\oto Y$ induces \emph{Kan adjunctions} \cite{Shen2013a} in $\QCat$ given by
\begin{equation} \label{Kan_def}
\bfig
\morphism/@{->}@<7pt>/<600,0>[\PY`\PX;\phi^{\od}]
\morphism(600,0)|b|/@{->}@<5pt>/<-600,0>[\PX`\PY;\phi_{\od}]
\place(290,0)[\bot]
\place(300,-220)[\phi^{\od}\tau=\tau\circ\phi,\quad\phi_{\od}\si=\si\lda\phi]
\place(1300,-100)[\text{and}]
\morphism(2000,0)/@{->}@<7pt>/<600,0>[\PdY`\PdX;\phi_{\opl}]
\morphism(2600,0)|b|/@{->}@<5pt>/<-600,0>[\PdX`\PdY;\phi^{\opl}]
\place(2290,0)[\bot]
\place(2300,-220)[\phi_{\opl}\tau=\phi\rda\tau,\quad\phi^{\opl}\si=\phi\circ\si.]
\efig
\end{equation}
Moreover, all the assignments in \eqref{graph_def} and \eqref{Kan_def} are 2-functorial, and one has two pairs of adjoint 2-functors \cite{Heymans2010} described by
\begin{equation} \label{QCat_QDist_adjunction}
\bfig
\morphism(-350,100)<400,0>[X`Y;\phi]
\place(-160,100)[\circ]
\morphism(-350,-100)<400,0>[Y`\PX;\olphi]
\morphism(-500,20)/-/<700,0>[`;]
\place(620,20)[\olphi y=\phi(-,y)]
\morphism(1600,0)/@{->}@<7pt>/<800,0>[\QCat`(\QDist)^{\op},;(-)^*]
\morphism(2400,0)|b|/@{->}@<5pt>/<-800,0>[(\QDist)^{\op},`\QCat;\sP]
\place(1950,0)[\bot]
\place(1500,-250)[(\phi^{\od}:\PY\to\PX)]
\morphism(2070,-250)/|->/<-150,0>[`;]
\place(2400,-250)[(\phi:X\oto Y)]
\morphism(-350,-500)<400,0>[X`Y;\phi]
\place(-160,-500)[\circ]
\morphism(-350,-700)<400,0>[X`\PdY;\ophi]
\morphism(-500,-580)/-/<700,0>[`;]
\place(620,-580)[\ophi x=\phi(x,-)]
\morphism(1600,-600)/@{->}@<7pt>/<800,0>[\QCat`(\QDist)^{\co},;(-)_*]
\morphism(2400,-600)|b|/@{->}@<5pt>/<-800,0>[(\QDist)^{\co},`\QCat;\sPd]
\place(1950,-600)[\bot]
\place(1450,-850)[(\phi^{\opl}:\PdX\to\PdY)]
\morphism(2070,-850)/|->/<-150,0>[`;]
\place(2400,-850)[(\phi:X\oto Y)]
\efig
\end{equation}
where ``$\co$'' refers to the dualization of 2-cells. The unit $\sy$ and the counit $\ep$ of the adjunction $(-)^*\dv\sP$ are respectively given by the Yoneda embeddings and their graphs:
$$\ep_X:=(\sy_X)_*:X\oto\PX.$$
The \emph{presheaf 2-monad} $\bbP=(\sP,\sfs,\sy)$ on $\QCat$ induced by $(-)^*\dv\sP$ sends each $\CQ$-functor $f:X\to Y$ to
$$f_!:=(f^*)^{\od}:\PX\to\PY,$$
which admits a right adjoint $f^!:=(f^*)_{\od}=(f_*)^{\od}:\PY\to\PX$ in $\QCat$; the monad multiplication $\sfs$ is given by
\begin{equation} \label{s_def}
\sfs_X=\ep_X^{\od}={\sup}_{\PX}=\sy_X^!:\PPX\to\PX,
\end{equation}
where ${\sup}_{\PX}=\sy_X^!$ is an immediate consequence of Lemma \ref{Yoneda}. Similarly, the unit $\syd$ is given by the co-Yoneda embeddings, and $\ep^{\dag}:=(\syd_{\Box})^*$ is the counit of the adjunction $(-)_*\dv\sPd$. The induced \emph{copresheaf 2-monad} $\bbP^{\dag}=(\sPd,\ssd,\syd)$ on $\QCat$ sends $f$ to
$$f_{\ie}:=(f_*)^{\opl}:\PdX\to\PdY,$$
which admits a left adjoint $f^{\ie}:=(f^*)^{\opl}=(f_*)_{\opl}:\PdY\to\PdX$ in $\QCat$, and the monad multiplication is given by
\begin{equation} \label{sd_def}
\ssd_X=(\ep_X^{\dag})^{\opl}={\inf}_{\PdX}=(\syd_X)^{\ie}:\PdPdX\to\PdX.
\end{equation}

We also point out that the presheaf 2-monad $\bbP$ on $\QCat$ is \emph{lax idempotent}, or \emph{of Kock-Z{\"o}berlein type} \cite{Stubbe2010}, in the sense that
$$(\sy_X)_!\leq\sy_{\PX}$$
for all $\CQ$-categories $X$. Dually, the copresheaf 2-monad $\bbP^{\dag}$ on $\QCat$ is \emph{oplax idempotent}, or \emph{of dual Kock-Z{\"o}berlein type}, in the sense that
$$\syd_{\PdX}\leq(\syd_X)_{\ie}$$
for all $\CQ$-categories $X$.

In a sequence of lemmata we now give a comprehensive list of rules that are needed in the calculations later on. While many of these rules are standard and well known, some are not and, in fact, new.

\begin{lem} \label{fully_faithful_presheaf} {\rm\cite{Shen2014,Shen2013a}}
Let $f:X\to Y$ be a $\CQ$-functor.
\begin{enumerate}[label={\rm (\arabic*)}]
\item \label{fully_faithful_presheaf:ff}
    $f$ is fully faithful $\iff$ $f^*\circ f_*=1_X^*$ $\iff$ $f^!\cdot f_!=1_{\PX}$ $\iff$ $f^{\ie}\cdot f_{\ie}=1_{\PdX}$ $\iff$ $f_!:\PX\to\PY$ is fully faithful $\iff$ $f_{\ie}:\PdX\to\PdY$ is fully faithful.
\item \label{fully_faithful_presheaf:es}
    If $f$ is essentially surjective, in the sense that, for any $y\in Y$,  there exists $x\in X$ with $y\cong fx$, then $f_*\circ f^*=1_Y^*$, $f_!\cdot f^!=1_{\PY}$, $f_{\ie}\cdot f^{\ie}=1_{\PdY}$ and both $f_!:\PX\to\PY$, $f_{\ie}:\PdX\to\PdY$ are surjective.
\end{enumerate}
\end{lem}


\begin{lem} \label{adjoint_functor} {\rm\cite{Pu2015,Stubbe2005}}
For all $\CQ$-functors $f:X\to Y$ and $g:Y\to X$,
\begin{align*}
f\dv g &\iff f_*=g^*\iff f^!=g_!\iff f_{\ie}=g^{\ie}\\
&\iff f_!\dv g_!\iff f^!\dv g^!\iff f_{\ie}\dv g_{\ie}\iff f^{\ie}\dv g^{\ie}.
\end{align*}
\end{lem}

\begin{lem} \label{functor_order}
For all $\CQ$-functors $f,g:X\to Y$ and $\CQ$-distributors $\phi,\psi:X\oto Y$,
\begin{enumerate}[label={\rm (\arabic*)}]
\item \label{f_leq_g}
$f\leq g\iff f_*\succeq g_*\iff f^*\preceq g^*\iff f_!\leq g_!\iff f_{\ie}\leq g_{\ie}\iff f^!\geq g^!\iff f^{\ie}\geq g^{\ie}.$
\item \label{phi_leq_psi}
$\phi\preceq\psi\iff\phi^{\od}\leq\psi^{\od}\iff\phi^{\opl}\geq\psi^{\opl}\iff\olphi\leq\olpsi\iff\ophi\geq\opsi$.
\end{enumerate}
\end{lem}

\begin{lem} \label{la_preserve_sup} {\rm\cite{Shen2016b,Stubbe2005}}
Let $f:X\to Y$ be a $\CQ$-functor between complete $\CQ$-categories. Then
$${\sup}_Y\cdot f_!\leq f\cdot{\sup}_X\quad\text{and}\quad f\cdot{\inf}_X\leq{\inf}_Y\cdot f_{\ie}.$$
Furthermore, $f$ is a left (resp. right) adjoint in $\QCat$ if, and only if, $\sup_Y\cdot f_!=f\cdot\sup_X$ (resp. $f\cdot\inf_X=\inf_Y\cdot f_{\ie}$).
\end{lem}

The above lemma shows that left (resp. right) adjoint $\CQ$-functors between complete $\CQ$-categories are exactly $\sup$-preserving (resp. $\inf$-preserving) $\CQ$-functors. Thus we denote the 2-subcategory of $\QCat$ consisting of separated complete $\CQ$-categories and $\sup$-preserving (resp. $\inf$-preserving) $\CQ$-functors by $\QSup$ (resp. $\QInf$).

\begin{lem} \label{dist_Yoneda}
The following identities hold for all $\CQ$-distributors $\phi:X\oto Y$.
\begin{enumerate}[label={\rm (\arabic*)}]
\item \label{dist_Yoneda:y}
    $\sy_X=\ola{1_X^*}$,\quad $\syd_X=\ora{1_X^*}$.
\item \label{dist_Yoneda:id}
    $1_{\PX}=\ola{(\sy_X)_*}$,\quad $1_{\PdX}=\ora{(\syd_X)^*}$.
\item \label{dist_Yoneda:olphi}
    $\olphi=\phi^{\od}\cdot\sy_Y$,\quad $\ophi=\phi^{\opl}\cdot\syd_X$.
\item \label{dist_Yoneda:phi}
    $\phi=\olphi^*\circ(\sy_X)_*=(\syd_Y)^*\circ\ophi_*$.
\item \label{dist_Yoneda:yphi}
    $(\sy_Y)_*\circ\phi=\phi^{\od *}\circ(\sy_X)_*$,\quad $\phi\circ(\syd_X)^*=(\syd_Y)^*\circ(\phi^{\opl})_*$.
\end{enumerate}
\end{lem}

\begin{proof}
(1), (3) are trivial, and (2), (4) are immediate consequences of the Yoneda lemma. For (5), note that the 2-functor
$$\sP:(\QDist)^{\op}\to\QCat,\quad(\phi:X\oto Y)\mapsto(\phi^{\od}:\PY\to\PX)$$
is faithful, and
$$((\sy_Y)_*\circ\phi)^{\od}=\phi^{\od}\cdot\sy_Y^!=\phi^{\od}\cdot{\sup}_{\PY}={\sup}_{\PX}\cdot(\phi^{\od})_!=\sy_X^!\cdot\phi^{\od*\od}=(\phi^{\od *}\circ(\sy_X)_*)^{\od}$$
follows by applying Lemma \ref{la_preserve_sup} to the left adjoint $\CQ$-functor $\phi^{\od}:\PY\to\PX$. The other identity can be verified analogously.
\end{proof}

\begin{lem} \label{functor_Yoneda}
The following identities hold for all $\CQ$-functors $f:X\to Y$.
\begin{enumerate}[label={\rm (\arabic*)}]
\item \label{functor_Yoneda:f_double}
    $f_{\ie !}=f^{\ie !}$,\quad $f_{!\ie}=f^{!\ie}$,\quad $(f_!)^!=(f^!)_!$,\quad $(f_{\ie})^{\ie}=(f^{\ie})_{\ie}$.
\item \label{functor_Yoneda:yd_nat}
    $\ola{f_*}=f^!\cdot\sy_Y$,\quad $\ora{f_*}=\syd_Y\cdot f=f_{\ie}\cdot\syd_X$.
\item  \label{functor_Yoneda:y_nat}
    $\ora{f^*}=f^{\ie}\cdot\syd_Y$,\quad $\ola{f^*}=\sy_Y\cdot f=f_!\cdot\sy_X$.
\item \label{functor_Yoneda:yf}
    $(\sy_X)_*\circ f^*=(f_!)^*\circ(\sy_Y)_*$,\quad $f_!\cdot\sy_X^!=\sy_Y^!\cdot f_{!!}$,\quad $(\sy_X)_{\ie}\cdot f^{\ie}=(f_!)^{\ie}\cdot(\sy_Y)_{\ie}$.
\item \label{functor_Yoneda:fyd}
    $f_*\circ(\syd_X)^*=(\syd_Y)^*\circ(f_{\ie})_*$,\quad $f_{\ie}\cdot(\syd_X)^{\ie}=(\syd_Y)^{\ie}\cdot f_{\ie\ie}$,\quad $(\syd_X)_!\cdot f^!=(f_{\ie})^!\cdot(\syd_Y)_!$.
\end{enumerate}
\end{lem}

\begin{proof}
For (1), $f_{\ie !}=f^{\ie !}$ since $(f^{\ie})_!\dv f^{\ie !}$ and $(f^{\ie})_!\dv f_{\ie !}$, and the other identities can be checked similarly. The non-trivial identities in (2) and (3) follow respectively from the naturality of $\syd$ and $\sy$, while (4) and (5) are immediate consequences of Lemma \ref{dist_Yoneda}\ref{dist_Yoneda:yphi}.
\end{proof}

\begin{lem} \label{transpose}
The following identities hold for all $\CQ$-distributors $\phi:X\oto Y$, $\psi:Y\oto Z$ and $\CQ$-functors $f$ whenever the operations make sense:
\begin{enumerate}[label={\rm (\arabic*)}]
\item \label{transpose:l}
    $\ola{\psi\circ\phi}=\phi^{\od}\cdot\olpsi=\sy_X^!\cdot\olphi_!\cdot\olpsi$,\quad $\ola{\psi\circ f^*}=f_!\cdot\olpsi$,\quad $\ola{f^*\circ\phi}=\olphi\cdot f$.
\item \label{transpose:r}
    $\ora{\psi\circ\phi}=\psi^{\opl}\cdot\ophi=(\syd_Z)^{\ie}\cdot\opsi_{\ie}\cdot\ophi$,\quad $\ora{\psi\circ f_*}=\opsi\cdot f$,\quad $\ora{f_*\circ\phi}=f_{\ie}\cdot\ophi$.
\end{enumerate}
\end{lem}

\begin{proof}
Straightforward calculations with the help of Lemmas \ref{dist_Yoneda} and \ref{functor_Yoneda}.
\end{proof}

\begin{lem} \label{fyx_leq_gyx}
For $\CQ$-functors $f,g:\PX\to Y$ (resp. $f,g:\PdX\to Y$), if $f$ (resp. $g$) is a left (resp. right) adjoint in $\QCat$, then
$$f\sy_X\leq g\sy_X\ (\text{resp.}\ f\syd_X\leq g\syd_X)\iff f\leq g.$$
\end{lem}

\begin{proof}
For the non-trivial direction, suppose that $f\dv h:Y\to\PX$, then $f\sy_X\leq g\sy_X$ implies $\sy_X\leq hg\sy_X$. Consequently, the Yoneda lemma and the $\CQ$-functoriality of $hg:\PX\to\PX$ imply
$$\si=(\sy_X)_*(-,\si)=1_{\PX}^*(\sy_X-,\si)\preceq 1_{\PX}^*(hg\sy_X-,hg\si)\preceq 1_{\PX}^*(\sy_X-,hg\si)=hg\si$$
and thus $\si\leq hg\si$, hence $f\si\leq g\si$ for all $\si\in\PX$.
%
\end{proof}

Since one already has the isomorphisms of ordered hom-sets
\begin{center}
\renewcommand\arraystretch{1.5}
\setlength{\tabcolsep}{1pt}
\begin{tabular}{ccccc}
$\QDist(X,Y)$ & $\cong$ & $\QCat(Y,\PX)$ & $\cong$ & $(\QCat)^{\co}(X,\PdY),$\\
$\phi$ & $\stackrel{\thicksim}{\longleftrightarrow}$ & $\olphi$ & $\stackrel{\thicksim}{\longleftrightarrow}$ & $\ophi$
\end{tabular}
\end{center}
with the adjunctions \eqref{QCat_QDist_adjunction} we obtain further isomorphisms in $\QSup$ and $\QInf$, as follows.

\begin{lem} \label{PX_PdY_la} {\rm\cite{Shen2013a}}
For all $\CQ$-categories $X,Y$, one has the natural isomorphisms of ordered hom-sets
\begin{align*}
\QDist(X,Y)&\cong(\QSup)^{\co}(\PX,\PdY)\cong\QInf(\PdY,\PX)\\
&\cong\QSup(\PY,\PX)\cong(\QInf)^{\co}(\PX,\PY).
\end{align*}
\end{lem}

\begin{proof}
Each $\CQ$-distributor $\phi:X\oto Y$ induces the Isbell adjunction $\uphi\dv\dphi:\PdY\to\PX$ \cite{Shen2013a} with
$$\uphi\si=\phi\lda\si\quad\text{and}\quad\dphi\tau=\tau\rda\phi$$
for all $\si\in\PX$, $\tau\in\PdY$. It is straightforward to check that
\begin{center}
\renewcommand\arraystretch{1.5}
\setlength{\tabcolsep}{1pt}
\begin{tabular}{ccccc}
$\QDist(X,Y)$ & $\cong$ & $(\QSup)^{\co}(\PX,\PdY)$ & $\cong$ & $\QInf(\PdY,\PX)$ \\
$\phi$ & $\stackrel{\thicksim}{\longleftrightarrow}$ & $\uphi$ & $\stackrel{\thicksim}{\longleftrightarrow}$ & $\dphi$\\
 & $\cong$ & $\QSup(\PY,\PX)$ & $\cong$ & $(\QInf)^{\co}(\PX,\PY)$\\
 & $\stackrel{\thicksim}{\longleftrightarrow}$ & $\phi^{\od}$ & $\stackrel{\thicksim}{\longleftrightarrow}$ & $\phi_{\od}$
\end{tabular}
\end{center}
gives the required isomorphisms. We refer to \cite[Theorems 4.4 \& 5.7]{Shen2013a} for details.
\end{proof}

\section{The non-discrete version of lax distributive laws and their lax algebras} \label{Lax_Dist_Laws}

In this section we establish the non-discrete version of the lax distributive laws considered in \cite{Tholen2016}. An equivalent framework in terms of lax extensions of 2-monads on $\QCat$ to $\QDist$ is presented below in Section \ref{Dist_Law_vs_Extension}.

For a 2-monad $\bbT=(T,m,e)$ on $\QCat$, a \emph{lax distributive law} $\lam:T\sP\to\sP T$ is given by a family
$$(\lam_X:T\PX\to\sP TX)_{X\in\ob(\QCat)}$$
of $\CQ$-functors satisfying the following inequalities for all $\CQ$-functors $f:X\to Y$:
\begin{enumerate}[label=(\alph*),leftmargin=0.7em,labelsep*=-0.3em]
\item \label{a}
\begin{tabular}{p{6.2cm}p{5cm}p{3.8cm}}
$\hskip 1cm\bfig
\square<600,300>[T\PX`T\PY`\sP TX`\sP TY;T(f_!)`\lam_X`\lam_Y`(Tf)_!]
\place(300,150)[\leq]
\efig$ & $(Tf)_!\cdot\lam_X\leq\lam_Y\cdot T(f_!)$ & (lax naturality of $\lam$);
\end{tabular}
\item \label{b}
\begin{tabular}{p{6.2cm}p{5cm}p{3.8cm}}
$\hskip 1cm\bfig
\Atriangle<300,300>[TX`T\PX`\sP TX;T\sy_X`\sy_{TX}`\lam_X]
\place(300,120)[\geq]
\efig$ & $\sy_{TX}\leq\lam_X\cdot T\sy_X$ & (lax $\bbP$-unit law);
\end{tabular}
\item \label{c}
\begin{tabular}{p{6.2cm}p{5cm}p{3.8cm}}
$\bfig
\square/`->`->`->/<1100,300>[T\PPX`\PP T X`T\PX`\sP T X;`T\sfs_X`\sfs_{T X}`\lam_X]
\morphism(0,300)<550,0>[T\PPX`\sP T\PX;\lam_{\PX}]
\morphism(550,300)<550,0>[\sP T\PX`\PP T X;(\lam_X)_!]
\place(550,130)[\geq]
\efig$ & $\sfs_{TX}\cdot(\lam_X)_!\cdot\lam_{\PX}\leq\lam_X\cdot T\sfs_X$ & (lax $\bbP$-mult. law);
\end{tabular}
\item \label{d}
\begin{tabular}{p{6.2cm}p{5cm}p{3.8cm}}
$\hskip 1cm\bfig
\Atriangle<300,300>[\PX`T\PX`\sP TX;e_{\PX}`(e_X)_!`\lam_X]
\place(300,120)[\geq]
\efig$ & $(e_X)_!\leq\lam_X\cdot e_{\PX}$ & (lax $\bbT$-unit law);
\end{tabular}
\item \label{e}
\begin{tabular}{p{6.2cm}p{5cm}p{3.8cm}}
$\bfig
\square/`->`->`->/<1100,300>[TT\PX`\sP TT X`T\PX`\sP T X;`m_{\PX}`(m_X)_!`\lam_X]
\morphism(0,300)<550,0>[TT\PX`T\sP T X;T\lam_X]
\morphism(550,300)<550,0>[T\sP T X`\sP TTX;\lam_{T X}]
\place(550,130)[\geq]
\efig$ & $(m_X)_!\cdot\lam_{TX}\cdot T\lam_X\leq\lam_X\cdot m_{\PX}$ & (lax $\bbT$-mult. law).
\end{tabular}
\end{enumerate}

Each of these laws is said to hold \emph{strictly} (at $f$ or $X$) if the respective inequality sign may be replaced by an equality sign; for a \emph{strict distributive law}, all lax laws must hold strictly everywhere. For simplicity, in what follows, we refer to a lax distributive law $\lam:T\sP\to T\sP$ just as a \emph{distributive law}
; we also say that $\bbT$ \emph{distributes} over $\bbP$ by $\lam$ in this case, adding \emph{strictly} when $\lam$ is strict.

Note that in the discrete case (see \cite{Tholen2016}), a distributive law $\lam$ of a monad $\bbT=(T,m,e)$ on $\Set/\CQ_0$ over the discrete presheaf monad $\bbP$ on $\Set/\CQ_0$ is usually required to be monotone, \emph{i.e.},
$$f\leq g\Lra\lam_X\cdot Tf\leq\lam_X\cdot Tg$$
for all $\CQ$-functors $f,g:Y\to\PX$. This property comes for free in the current non-discrete case since the 2-functor $T$ of the 2-monad $\bbT=(T,m,e)$ will respect the order.


\begin{defn} \label{lam_alg_def}
For a distributive law $\lam:T\sP\to\sP T$, a \emph{lax $\lam$-algebra} $(X,p)$ over $\CQ$ is a $\CQ$-category $X$ with a $\CQ$-functor $p:TX\to\PX$ satisfying
\begin{enumerate}[label=(\alph*),start=6,leftmargin=1em,labelsep*=-0.3em]
\item \label{f}
\begin{tabular}{p{6.2cm}p{5cm}p{3.8cm}}
\hskip 1.2cm $\bfig
\Atriangle<300,300>[X`TX`\PX;e_X`\sy_X`p]
\place(300,120)[\geq]
\efig$ & $\sy_X\leq p\cdot e_X$ & (lax unit law);
\end{tabular}
\item \label{g}
\begin{tabular}{p{6.2cm}p{5cm}p{3.8cm}}
$\bfig
\square/`->`->`->/<1200,300>[TTX`\PPX`TX`\PX;`m_X`\sfs_X`p]
\morphism(0,300)<400,0>[TTX`T\PX;Tp]
\morphism(400,300)<400,0>[T\PX`\sP TX;\lam_X]
\morphism(800,300)<400,0>[\sP TX`\PPX;p_!]
\place(600,130)[\geq]
\efig$ & $\sfs_X\cdot p_!\cdot\lam_X\cdot Tp\leq p\cdot m_X$ & (lax mult. law).
\end{tabular}
\end{enumerate}

A \emph{lax $\lam$-homomorphism} $f:(X,p)\to(Y,q)$ of lax $\lam$-algebras is a $\CQ$-functor $f:X\to Y$ which satisfies
\begin{enumerate}[label=(\alph*),start=8,leftmargin=1em,labelsep*=-0.3em]
\item \label{h}
\begin{tabular}{p{6.2cm}p{3.2cm}p{4.8cm}}
\hskip 0.8cm $\bfig
\square<800,300>[TX`TY`\PX`\PY;Tf`p`q`f_!]
\place(400,130)[\leq]
\efig$ & $f_!\cdot p\leq q\cdot Tf$ & (lax homomorphism law).
\end{tabular}
\end{enumerate}

The resulting 2-category is denoted by $\lamQAlg$, with the local order inherited from $\QCat$.
\end{defn}

\begin{prop}
$\lamQAlg$ is topological over $\QCat$ and, hence, totally complete and totally cocomplete.
\end{prop}

\begin{proof}
For any family of $\lam$-algebras $(Y_j,q_j)$ and $\CQ$-functors $f_j:X\to Y_j$ $(j\in J)$,
$$p:=\bw_{j\in J}(f_j)^!\cdot q_j\cdot Tf_j$$
gives the initial structure on $X$ with respect to the forgetful functor $\lamQAlg\to\QCat$, and thus establishes the topologicity of $\lamQAlg$ over $\QCat$ (see \cite{Adamek1990}). The total completeness and total cocompleteness of $\lamQAlg$ then follow from the respective properties of $\QCat$ (see \cite[Theorem 2.7]{Shen2015}).
\end{proof}

\section{Flat distributive laws of 2-monads on $\QCat$ over $\bbP$} \label{Flat_Dist_Laws}

Given a 2-monad $\bbT=(T,m,e)$ on $\QCat$, a distributive law $\lam$ of $\bbT$ over $\bbP$ is called \emph{flat} (or \emph{normal}) if it satisfies the $\bbP$-unit law \ref{b} strictly; that is, if
$$\sy_{TX}=\lam_X\cdot T\sy_X$$
for all $X\in\ob(\QCat)$. In fact, although there may be several distributive laws of $\bbT$ over $\bbP$, only one of them may be flat:

\begin{prop} \label{flat_Tyx}
A flat distributive law $\lam:T\sP\to\sP T$ must satisfy
$$\lam_X=\ola{(T\sy_X)_*}:T\PX\to\sP TX.$$
\end{prop}

\begin{proof}
Indeed, $\lam_X\geq\ola{(T\sy_X)_*}$ holds for any $\lam:T\sP\to\sP T$ satisfying the lax laws \ref{b} and \ref{c}, since
\begin{align*}
\lam_X&=\lam_X\cdot T\ola{(\sy_X)_*}&(\text{Lemma \ref{dist_Yoneda}\ref{dist_Yoneda:id}})\\
&=\lam_X\cdot T\sy_X^!\cdot T\sy_{\PX}&(\text{Lemma \ref{functor_Yoneda}\ref{functor_Yoneda:yd_nat}})\\
&\geq\sy_{TX}^!\cdot(\lam_X)_!\cdot\lam_{\PX}\cdot T\sy_{\PX}&(\lam\ \text{satisfies \ref{c}})\\
&\geq(T\sy_X)^!\cdot\lam_X^!\cdot(\lam_X)_!\cdot\sy_{T\PX}&(\lam\ \text{satisfies \ref{b}})\\
&\geq(T\sy_X)^!\cdot\sy_{T\PX}&((\lam_X)_!\dv\lam_X^!)\\
&=\ola{(T\sy_X)_*}.&(\text{Lemma \ref{functor_Yoneda}\ref{functor_Yoneda:yd_nat}})
\end{align*}
When $\lam$ also satisfies the laws \ref{a} laxly and \ref{b} strictly, one has
\begin{align*}
\lam_X&\leq(T\sy_X)^!\cdot(T\sy_X)_!\cdot\lam_X&((T\sy_X)_!\dv(T\sy_X)^!)\\
&\leq(T\sy_X)^!\cdot\lam_{\PX}\cdot T(\sy_X)_!&(\lam\ \text{satisfies \ref{a}})\\
&\leq(T\sy_X)^!\cdot\lam_{\PX}\cdot T\sy_{\PX}&(\bbP\ \text{is lax idempotent})\\
&=(T\sy_X)^!\cdot\sy_{T\PX}&(\lam\ \text{satisfies \ref{b} strictly})\\
&=\ola{(T\sy_X)_*}.&(\text{Lemma \ref{functor_Yoneda}\ref{functor_Yoneda:yd_nat}})
\end{align*}
\end{proof}

The following Theorem provides the crucial tool for establishing the flat distributive laws over $\bbP$ presented in Sections \ref{Law_P}--\ref{Law_PdP}. In fact, it just paraphrases a lax extension result of \cite{Lai2016a} for 2-monads on $\QCat$  to $\QDist$, which we formulate explicitly below as Corollary \ref{Yoneda_flat_dist_extension} and which, in turn, builds on a lax extension result for endofunctors of $\VCat$ (where $\sV$ is a quantale) used in \cite{Akhvlediani2010}.

\begin{thm} \label{Yoneda_flat_dist_law}
Let $\bbT=(T,m,e)$ be a 2-monad on $\QCat$ such that $T\sy_X:TX\to T\PX$ is fully faithful for all $\CQ$-categories $X$. Then
$$\lam_X=\ola{(T\sy_X)_*}:T\PX\to\sP TX$$
defines a flat distributive law of $\bbT$ over $\bbP$, and it is the only one.
\end{thm}

\begin{proof}
We verify that for any 2-monad $\bbT=(T,m,e)$, $\lam$ satisfies the lax laws \ref{a}--\ref{d}, with \ref{c} always holding strictly; meanwhile, \ref{b} holds strictly if, and only if, $T\sy_X$ is fully faithful, in which case $\lam$ also satisfies the lax law \ref{e}; this will complete the proof, by Proposition \ref{flat_Tyx}.

\ref{a} $(Tf)_!\cdot\lam_X\leq\lam_Y\cdot T(f_!)$. From the naturality of $\sy$ one has
$$(T(f_!))_*\circ(T\sy_X)_*=(T\sy_Y)_*\circ(Tf)_*,$$
and then the adjunction rules ensure that
$$(Tf)_!\cdot\lam_X=\ola{(T\sy_X)_*\circ(Tf)^*}\leq\ola{(T(f_!))^*\circ(T\sy_Y)_*}=\lam_Y\cdot T(f_!),$$
where the first and the last equalities follow from Lemma \ref{transpose}\ref{transpose:l}.

\ref{b} $\sy_{TX}\leq\lam_X\cdot T\sy_X$.
  The adjunction rules give $$\sy_{TX}=\ola{1_{TX}^*}\leq\ola{(T\sy_X)^*\circ(T\sy_X)_*}=\lam_X\cdot T\sy_X,$$ where ``$\leq$'' may be replaced by ``$=$'' if, and only if,
$T\sy_X$ is fully faithful, by Lemma \ref{fully_faithful_presheaf}\ref{fully_faithful_presheaf:ff}.

\ref{c} $\sfs_{TX}\cdot(\lam_X)_!\cdot\lam_{\PX}=\lam_X\cdot T\sfs_X$. Since the 2-functor $T$ preserves adjunctions in $\QCat$, one has
\begin{align*}
\sfs_{TX}\cdot(\lam_X)_!\cdot\lam_{\PX}&=\ola{(T\sy_{\PX})_*\circ(T\sy_X)_*}&(\text{Lemma \ref{transpose}\ref{transpose:l}})\\
&=\ola{(T(\sy_X)_!)_*\circ(T\sy_X)_*}&(\sy\ \text{is natural})\\
&=\ola{(T(\sy_X)^!)^*\circ(T\sy_X)_*}&(T(\sy_X)_!\dv T(\sy_X)^!)\\
&=\lam_X\cdot T\sfs_X.&(\text{Lemma \ref{transpose}\ref{transpose:l}})
\end{align*}

\ref{d} $(e_X)_!\leq\lam_X\cdot e_{\PX}$. The naturality of $e$ induces $T\sy_X\cdot e_X=e_{\PX}\cdot\sy_X$, and consequently
\begin{align*}
(e_X)_!&\leq(T\sy_X)^!\cdot(e_{\PX})_!\cdot(\sy_X)_!&((T\sy_X)_!\dv(T\sy_X)^!)\\
&\leq(T\sy_X)^!\cdot(e_{\PX})_!\cdot\sy_{\PX}&(\bbP\ \text{is lax idempotent})\\
&=(T\sy_X)^!\cdot\sy_{T\PX}\cdot e_{\PX}&(\sy\ \text{is natural})\\
&=\lam_X\cdot e_{\PX}. &(\text{Lemma \ref{functor_Yoneda}\ref{functor_Yoneda:yd_nat}})
\end{align*}

\ref{e} $(m_X)_!\cdot\lam_{TX}\cdot T\lam_X\leq\lam_X\cdot m_{\PX}$ if \ref{b} holds strictly. From the naturality of $m$ one has $T\sy_X\cdot m_X=m_{\PX}\cdot TT\sy_X$, and it follows that
\begin{align*}
(m_X)_!\cdot\lam_{TX}\cdot T\lam_X&\leq(T\sy_X)^!\cdot(m_{\PX})_!\cdot(TT\sy_X)_!\cdot\lam_{TX}\cdot T\lam_X&((T\sy_X)_!\dv(T\sy_X)^!)\\
&\leq(T\sy_X)^!\cdot(m_{\PX})_!\cdot\lam_{T\PX}\cdot T(T\sy_X)_!\cdot T\lam_X&(\lam\ \text{satisfies \ref{a}})\\
&\leq(T\sy_X)^!\cdot(m_{\PX})_!\cdot\lam_{T\PX}\cdot T\lam_{\PX}\cdot TT(\sy_X)_!&(\lam\ \text{satisfies \ref{a}})\\
&\leq(T\sy_X)^!\cdot(m_{\PX})_!\cdot\lam_{T\PX}\cdot T\lam_{\PX}\cdot TT\sy_{\PX}&(\bbP\ \text{is lax idempotent})\\
&=(T\sy_X)^!\cdot(m_{\PX})_!\cdot\lam_{T\PX}\cdot T\sy_{T\PX}&(\lam\ \text{satisfies \ref{b} strictly})\\
&=(T\sy_X)^!\cdot(m_{\PX})_!\cdot\sy_{TT\PX}&(\lam\ \text{satisfies \ref{b} strictly})\\
&=(T\sy_X)^!\cdot\sy_{T\PX}\cdot m_{\PX}&(\sy\ \text{is natural})\\
&=\lam_X\cdot m_{\PX}.&(\text{Lemma \ref{functor_Yoneda}\ref{functor_Yoneda:yd_nat}})
\end{align*}
\end{proof}

\section{The distributive laws of the (co)presheaf 2-monad} \label{Law_P}

As an immediate consequence of Lemma \ref{fully_faithful_presheaf} and Theorem \ref{Yoneda_flat_dist_law}, one sees that
$$\lam_X=\ola{(\sy_X)_{! *}}=((\sy_X)_!)^!\cdot\sy_{\PPX}=\sy_{\PX}\cdot\sy_X^!=\sy_{\PX}\cdot{\sup}_{\PX}:\PPX\to\PPX$$
defines a flat distributive law of the presheaf 2-monad $\bbP$ over itself:

\begin{prop} \label{P_dist}
The presheaf 2-monad $\bbP$ distributes flatly over itself by $\lam$ with
$$\lam_X=\sy_{\PX}\cdot{\sup}_{\PX}:\PPX\to\PPX.$$
\end{prop}

Now we describe the lax algebras for this distributive law. A \emph{$\CQ$-closure space} \cite{Shen2016b,Shen2013a} is a pair $(X,c)$ that consists of a $\CQ$-category $X$ and a \emph{$\CQ$-closure operation} $c$ on $\PX$; that is, a $\CQ$-functor $c:\PX\to\PX$ satisfying $1_{\PX}\leq c$ and $c\cdot c=c$. A \emph{continuous $\CQ$-functor} $f:(X,c)\to(Y,d)$ between $\CQ$-closure spaces is a $\CQ$-functor $f:X\to Y$ such that
$$f_!\cdot c\leq d\cdot f_!:\PX\to\PY.$$

$\CQ$-closure spaces and continuous $\CQ$-functors constitute the 2-category $\QCls$, with the local order inherited from $\QCat$. One sees quite easily that these are the lax $\lam$-algebras over $\CQ$:

\begin{thm} \label{P_Alg}
$\lamQAlg\cong\QCls$.
\end{thm}

\begin{proof}
For any $\CQ$-category $X$, we show that a $\CQ$-functor $c:\PX\to\PX$ gives a lax $\lam$-algebra structure on $X$ if, and only if, $(X,c)$ is a $\CQ$-closure space.

$c$ satisfies \ref{f} $\iff 1_{\PX}\leq c$: This is an immediate consequence of Lemma \ref{fyx_leq_gyx}.

$c$ satisfies \ref{g} $\iff c\cdot c\leq c$: Note that
\begin{align*}
c\cdot c&={\sup}_{\PX}\cdot\sy_{\PX}\cdot c\cdot{\sup}_{\PX}\cdot\sy_{\PX}\cdot c\\
&={\sup}_{\PX}\cdot c_!\cdot\sy_{\PX}\cdot{\sup}_{\PX}\cdot c_!\cdot\sy_{\PX}&(\sy\ \text{is natural})\\
&={\sup}_{\PX}\cdot c_!\cdot\lam_X\cdot c_!\cdot\sy_{\PX},
\end{align*}
and thus
\begin{align*}
c\cdot c\leq c&\iff{\sup}_{\PX}\cdot c_!\cdot\lam_X\cdot c_!\cdot\sy_{\PX}\leq c\\
&\iff{\sup}_{\PX}\cdot c_!\cdot\lam_X\cdot c_!\leq c\cdot{\sup}_{\PX},&({\sup}_{\PX}\dv\sy_{\PX})
\end{align*}
which is precisely the condition \ref{g}.

Therefore, the isomorphism between $\lamQAlg$ and $\QCls$ follows since a continuous $\CQ$-functor $f:(X,c)\to(Y,d)$ is exactly a $\CQ$-functor $f:X\to Y$ satisfying the condition \ref{h}.
\end{proof}

The distributive law of the copresheaf 2-monad $\bbP^{\dag}$ over $\bbP$ arising from Lemma \ref{fully_faithful_presheaf} and Theorem \ref{Yoneda_flat_dist_law} is even strict:

\begin{prop} \label{Pd_dist}
The copresheaf 2-monad $\bbP^{\dag}$ distributes strictly over $\bbP$ by $\lamd$ with
$$\lamd_X=\ola{(\sy_X)_{\ie *}}=((\sy_X)_{\ie})^!\cdot\sy_{\PdPX}:\PdPX\to\PPdX.$$
\end{prop}

\begin{proof}
One needs to show that $\lamd$ satisfies the laws \ref{a}--\ref{e} strictly. We leave the lengthy but routine verification to the readers.
\end{proof}

\begin{rem}
Stubbe described a strict distributive law of $\bbP$ over $\bbP^{\dag}$ given by
\begin{equation} \label{Pd_dist_Stubbe}
\PPdX\to^{(\sy_X)_{\ie !}}\PPd\PX\to^{\sup_{\PdPX}}\PdPX
\end{equation}
in \cite{Stubbe2017}. In fact, the strict distributive law $\lamd_X:\PdPX\to\PPdX$ defined in Proposition \ref{Pd_dist} is precisely the right adjoint of \eqref{Pd_dist_Stubbe} in $\QCat$.
\end{rem}

Recall that a $\CQ$-category is a monad in $\QRel$. Similarly, a monad in $\QDist$ gives ``a $\CQ$-category over a base $\CQ$-category''; that is, a $\CQ$-category $X$ equipped with a $\CQ$-distributor $\al:X\oto X$, such that $1_X^*\preceq\al$ and $\al\circ\al\preceq\al$. The latter two inequalities actually force the $\CQ$-relation $\al$ on $X$ to be a $\CQ$-distributor, since with $a=1_X^*$ one has
$$a\circ(\al\circ a)\preceq a\circ(\al\circ\al)\preceq a\circ\al\preceq\al\circ\al\preceq\al.$$
Thus, a monad in $\QDist$ is given by a set $X$ over $\CQ_0$ that comes equipped with two $\CQ$-category structures, comparable by ``$\preceq$''. With morphisms to laxly preserve both structures we obtain the 2-category $\Mon(\QDist)$; hence, its morphisms $f:(X,\al)\to(Y,\be)$ are precisely the $\CQ$-functors $f:X\to Y$ with
$$f_!\cdot\olal\leq\olbe\cdot f$$
or, equivalently, $\al(x,x')\preceq\be(fx,fx')$ for all $x,x'\in X$, which are equipped with the order inherited from $\QCat$.

\begin{rem}
The 2-category $\Mon(\QDist)$ must be carefully distinguished from the 2-categories ${\sf Mnd}(\QDist)$, as  considered by Street \cite{Street1972}, and ${\sf EM}(\QDist)$, as considered by Lack and Street \cite{Lack2002}. Although all three 2-categories have the same objects, they have different 1-cells. In fact, ${\sf Mnd}(\QDist)$ and ${\sf EM}(\QDist)$ are internally constructed from $\QDist$, with 1-cells being those of $\QDist$, i.e., $\CQ$-distributors.
\end{rem}


\begin{thm} \label{Pd_Alg}
$\lamdQAlg\cong\Mon(\QDist)$.
\end{thm}

\begin{proof}
{\bf Step 1.} We show that, if $(X,p)$ is a $(\lamd,\CQ)$-algebra, then
\begin{equation} \label{lamdQ_p}
p={\inf}_{\PX}\cdot p_{\ie}\cdot(\syd_X)_{\ie}.
\end{equation}
Indeed, the conditions \ref{f} and \ref{g} for the $(\lamd,\CQ)$-algebra $(X,p)$ read as
\begin{enumerate}[label=(\alph*),start=6]
\item \label{lamd:f}
    $\sy_X\leq p\cdot\syd_X$ and
\item \label{lamd:g}
    $\sy_X^!\cdot p_!\cdot\lamd_X\cdot p_{\ie}\leq p\cdot(\syd_X)^{\ie}$,
\end{enumerate}
and consequently
\begin{align*}
p&={\inf}_{\PX}\cdot\syd_{\PX}\cdot p\\
&={\inf}_{\PX}\cdot p_{\ie}\cdot\syd_{\PdX}&(\syd\ \text{is natural})\\
&\leq{\inf}_{\PX}\cdot p_{\ie}\cdot(\syd_X)_{\ie}&(\bbP^{\dag}\ \text{is oplax idempotent})\\
&=(\syd_X)^!\cdot(\syd_X)_!\cdot{\inf}_{\PX}\cdot p_{\ie}\cdot(\syd_X)_{\ie}&(\syd_X\ \text{is fully faithful})\\
&=(\syd_X)^!\cdot\lamd_X\cdot\syd_{\PX}\cdot{\inf}_{\PX}\cdot p_{\ie}\cdot(\syd_X)_{\ie}&(\lamd\ \text{satisfies \ref{d} strictly})\\
&\leq(\syd_X)^!\cdot\lamd_X\cdot p_{\ie}\cdot(\syd_X)_{\ie}&(\syd_{\PX}\dv{\inf}_{\PX})\\
&\leq(\syd_X)^!\cdot p^!\cdot p_!\cdot\lamd_X\cdot p_{\ie}\cdot(\syd_X)_{\ie}&(p_!\dv p^!)\\
&\leq\sy_X^!\cdot p_!\cdot\lamd_X\cdot p_{\ie}\cdot(\syd_X)_{\ie}&(p\ \text{satisfies \ref{lamd:f}})\\
&\leq p\cdot(\syd_X)^{\ie}\cdot(\syd_X)_{\ie}&(p\ \text{satisfies \ref{lamd:g}})\\
&=p.&(\syd_X\ \text{is fully faithful})
\end{align*}

{\bf Step 2.} As an immediate consequence of \eqref{lamdQ_p}, $p$ is a right adjoint in $\QCat$. For any $\CQ$-category $X$, one already has
$$\QDist(X,X)\cong\QInf(\PdX,\PX)$$
from Lemma \ref{PX_PdY_la}, with the isomorphism given by
$$(\al:X\oto X)\mapsto(\dal:\PdX\to\PX,\quad\dal\tau=\tau\rda\al).$$
Hence, in order for us to establish a bijection between monads on $X$ (in $\QDist$) and $(\lamd,\CQ)$-algebra structures on $X$, it suffices to prove
\begin{enumerate}[label=$\bullet$]
\item $1_X^*\preceq\al\iff\dal$ satisfies \ref{lamd:f}, and
\item $\al\circ\al\preceq\al\iff\dal$ satisfies \ref{lamd:g}
\end{enumerate}
for all $\CQ$-distributors $\al:X\oto X$.

First, $1_X^*\preceq\al\iff\dal$ satisfies \ref{lamd:f}. Since $\ola{1_X^*}=\sy_X$ and, as one easily sees, $\olal=\dal\cdot\syd_X$, the equivalence $1_X^*\preceq\al\iff\sy_X\leq\dal\cdot\syd_X$ follows immediately.

Second, $\al\circ\al\preceq\al\iff\dal$ satisfies \ref{lamd:g}, \emph{i.e.},
$$\sy_X^!\cdot(\dal)_!\cdot\lamd_X\cdot(\dal)_{\ie}\leq\dal\cdot(\syd_X)^{\ie}=\dal\cdot{\inf}_{\PdX}.$$
Note that
\begin{align*}
\ola{\al\circ\al}&=\sy_X^!\cdot\olal_!\cdot\olal&(\text{Lemma \ref{transpose}\ref{transpose:l}})\\
&=\sy_X^!\cdot(\dal)_!\cdot(\syd_X)_!\cdot\dal\cdot\syd_X &(\dal\cdot\syd_X=\olal)\\
&=\sy_X^!\cdot(\dal)_!\cdot\lamd_X\cdot\syd_{\PX}\cdot\dal\cdot\syd_X&(\lamd\ \text{satisfies \ref{d} strictly})\\
&=\sy_X^!\cdot(\dal)_!\cdot\lamd_X\cdot(\dal)_{\ie}\cdot\syd_{\PdX}\cdot\syd_X&(\syd\ \text{is natural})
\end{align*}
and, hence,
\begin{align*}
\al\circ\al\preceq\al&\iff\ola{\al\circ\al}\leq\olal=\dal\cdot\syd_X\\
&\iff\sy_X^!\cdot(\dal)_!\cdot\lamd_X\cdot(\dal)_{\ie}\cdot\syd_{\PdX}\cdot\syd_X\leq\dal\cdot\syd_X\\
&\iff\sy_X^!\cdot(\dal)_!\cdot\lamd_X\cdot(\dal)_{\ie}\cdot\syd_{\PdX}\leq\dal=\dal\cdot{\inf}_{\PdX}\cdot\syd_{\PdX}&(\text{Lemma \ref{fyx_leq_gyx}})\\
&\iff\sy_X^!\cdot(\dal)_!\cdot\lamd_X\cdot(\dal)_{\ie}\leq\dal\cdot{\inf}_{\PdX},&(\text{Lemma \ref{fyx_leq_gyx}})
\end{align*}
as desired.

{\bf Step 3.} $f:(X,\al)\to(Y,\be)$ is a morphism in $\Mon(\QDist)$ if, and only if, $f:(X,\dal)\to(Y,\dbe)$ satisfies \ref{h}. Indeed,
\begin{align*}
f_!\cdot\dal\leq\dbe\cdot f_{\ie}&\iff f_!\cdot\dal\cdot\syd_X\leq\dbe\cdot f_{\ie}\cdot\syd_X&(\text{Lemma \ref{fyx_leq_gyx}})\\
&\iff f_!\cdot\dal\cdot\syd_X\leq\dbe\cdot\syd_Y\cdot f&(\syd\ \text{is natural})\\
&\iff f_!\cdot\olal\leq\olbe\cdot f,
\end{align*}
which completes the proof.
\end{proof}

\section{The distributive law of the double presheaf 2-monad} \label{Law_PPd}

Recall that the adjunctions $(-)^*\dv\sP$ and $(-)_*\dv\sPd$ displayed in \eqref{QCat_QDist_adjunction} give rise to the isomorphisms
\begin{equation} \label{YPX_XY_XPdY}
\QCat(Y,\PX)\cong\QDist(X,Y)\cong\QCat(X,\PdY),
\end{equation}
for all $\CQ$-categories $X$, $Y$. In fact, \eqref{YPX_XY_XPdY} induces another pair of adjoint 2-functors \cite{Stubbe2017}
\begin{equation} \label{Pdc_dv_Pd}
\sPd_{\rc}\dv\sP_{\rc}:\QCat\to(\QCat)^{\co\op},
\end{equation}
which map objects as $\sPd$ and $\sP$ do, but with $\sPd_{\rc} f=f^{\ie}$ and $\sP_{\rc} f=f^!$ for all $\CQ$-functor $f$. The units and counits of this adjunction are respectively given by
$$\begin{array}{ccc}
\sy_{\PdX}\cdot\syd_X=(\syd_X)_!\cdot\sy_X:X\to\PPdX & \text{and} & \syd_{\PX}\cdot\sy_X=(\sy_X)_{\ie}\cdot\syd_X:X\to\PdPX\\
\bfig \square<600,350>[X`\PdX`\PX`\PPdX;\syd_X`\sy_X`\sy_{\PdX}`(\syd_X)_!] \efig & &
\bfig \square<600,350>[X`\PX`\PdX`\PdPX;\sy_X`\syd_X`\syd_{\PX}`(\sy_X)_{\ie}] \efig
\end{array}$$
for all $\CQ$-categories $X$. This adjunction induces the \emph{double presheaf 2-monad} $(\sP_{\rc}\sPd_{\rc},\Fy,\Fs)$ on $\QCat$ with the multiplication given by
\begin{equation} \label{Fs_def}
\Fs_X=((\sy_{\PdX})_{\ie}\cdot\syd_{\PdX})^!=(\syd_{\PPdX}\cdot\sy_{\PdX})^!=\sfs_{\PdX}\cdot(\syd_{\PPdX})^!:\PPd\PPdX\to\PPdX.
\end{equation}
As Lemma \ref{functor_Yoneda}\ref{functor_Yoneda:f_double} implies $\sP_{\rc}\sPd_{\rc}=\PPd$, the double presheaf 2-monad on $\QCat$ may alternatively be written as
$$\bbP\bbP^{\dag}=(\PPd,\Fy,\Fs).$$
With Lemma \ref{fully_faithful_presheaf} and Theorem \ref{Yoneda_flat_dist_law} one obtains immediately:

\begin{prop} \label{PPd_dist}
The double presheaf 2-monad $\bbP\bbP^{\dag}$ distributes flatly over $\bbP$ by $\Lam$, with
$$\Lam_X=\ola{((\sy_X)_{\ie !})_*}=((\sy_X)_{\ie !})^!\cdot\sy_{\PPd\PX}=\sy_{\PPdX}\cdot((\sy_X)_{\ie})^!:\PPd\PX\to\sP\PPdX.$$
\end{prop}

A \emph{$\CQ$-interior space} is a pair $(X,c)$ consisting of a $\CQ$-category $X$ and a $\CQ$-closure operation $c$ on $\PdX$. A \emph{continuous $\CQ$-functor} $f:(X,c)\to(Y,d)$ between $\CQ$-interior spaces is a $\CQ$-functor $f:X\to Y$ such that
$$c\cdot f^{\ie}\leq f^{\ie}\cdot d:\PdY\to\PdX.$$

$\CQ$-interior spaces and continuous $\CQ$-functors constitute a 2-category $\QInt$, with the local order inherited from $\QCat$. To prove that these are precisely the lax $\Lam$-algebras over $\CQ$ requires the full arsenal of tools provided in this paper.

\begin{rem}
When $\CQ$ is a commutative quantale, $\sV$, one has $u\lda v=v\rda u$ for all $u,v\in\sV$. Considering a set $X$ as a discrete $\sV$-category one can display $\PX$ and $\PdX$ as having the same underlying set $\sV^X$, and for all $\phi,\psi\in\sV^X$ one has
$$1_{\PX}^*(\phi,\psi)=1_{\PdX}^*(\psi,\phi),$$
\emph{i.e.}, $\PdX$ is the dual of $\PX$. Thus, for a closure operation $c:\PdX\to\PdX$ one has
$$1_{\PdX}\leq c\iff c\leq 1_{\PX},$$
that is, $c$ is an \emph{interior operation} on $\PX$ (see \cite{Lai2009}).
Particularly, when $\sV={\bf 2}$, $\PX$ is just the powerset of $X$, and a closure operation $c$ on $\PdX$ is exactly an interior operation on the powerset of $X$. So, an interior space $(X,c)$ as defined here coincides with the usual notion.
\end{rem}

\begin{thm} \label{PPd_Alg}
$\LamQAlg\cong\QInt$.
\end{thm}

\begin{proof}
{\bf Step 1.} We show that, if $(X,p)$ is a $(\Lam,\CQ)$-algebra, then
\begin{equation} \label{LamQ_p}
p=({\inf}_{\PPdX}\cdot p^{\ie}\cdot\syd_{\PX}\cdot\sy_X)^!\cdot\sy_{\PPdX}=(\sy_X)^!\cdot(\syd_{\PX})^!\cdot p_{\ie !}\cdot{\inf}_{\PPdX}^!\cdot\sy_{\PPdX}.
\end{equation}
Indeed, from the definition of the 2-monad $\bbP\bbP^{\dag}$ one may translate the conditions \ref{f} and \ref{g} for $(X,p)$ respectively as
$$\sy_X\leq p\cdot(\syd_X)_!\cdot\sy_X\quad\text{and}\quad\sy_X^!\cdot p_!\cdot\Lam_X\cdot p_{\ie !}\leq p\cdot{\sup}_{\PPdX}\cdot(\syd_{\PPdX})^!.$$
Since from Lemma \ref{fyx_leq_gyx} one has
$$\sy_X\leq p\cdot(\syd_X)_!\cdot\sy_X\iff 1_{\PX}\leq p\cdot(\syd_X)_!$$
and since $\Lam_X=\sy_{\PPdX}\cdot((\sy_X)_{\ie})^!$ implies
\begin{align*}
\sy_X^!\cdot p_!\cdot\Lam_X\cdot p_{\ie !}&=\sy_X^!\cdot p_!\cdot\sy_{\PPdX}\cdot((\sy_X)_{\ie})^!\cdot p_{\ie !}\\
&={\sup}_{\PX}\cdot\sy_{\PX}\cdot p\cdot((\sy_X)_{\ie})^!\cdot p_{\ie !}&(\sy\ \text{is natural})\\
&=p\cdot((\sy_X)_{\ie})^!\cdot p_{\ie !},
\end{align*}
the conditions \ref{f} and \ref{g} may be simplified to read as
\begin{enumerate}[label=(\alph*),start=6]
\item \label{Lam:f}
    $1_{\PX}\leq p\cdot(\syd_X)_!$ and
\item \label{Lam:g}
    $p\cdot((\sy_X)_{\ie})^!\cdot p_{\ie !}\leq p\cdot{\sup}_{\PPdX}\cdot(\syd_{\PPdX})^!$.
\end{enumerate}
Therefore,
\begin{align*}
p&={\sup}_{\PX}\cdot\sy_{\PX}\cdot p\\
&=(\sy_X)^!\cdot p_!\cdot\sy_{\PPdX}&(\sy\ \text{is natural})\\
&=(\sy_X)^!\cdot p_!\cdot({\inf}_{\PPdX})_!\cdot{\inf}_{\PPdX}^!\cdot\sy_{\PPdX}&({\inf}_{\PPdX}\ \text{is surjective})\\
&\leq(\sy_X)^!\cdot({\inf}_{\PX})_!\cdot p_{\ie !}\cdot{\inf}_{\PPdX}^!\cdot\sy_{\PPdX}&(\text{Lemma \ref{la_preserve_sup}})\\
&=(\sy_X)^!\cdot(\syd_{\PX})^!\cdot p_{\ie !}\cdot{\inf}_{\PPdX}^!\cdot\sy_{\PPdX}&(\syd_{\PX}\dv{\inf}_{\PX})\\
&=(\syd_X)^!\cdot((\sy_X)_{\ie})^!\cdot p_{\ie !}\cdot{\inf}_{\PPdX}^!\cdot\sy_{\PPdX}&(\syd\ \text{is natural})\\
&\leq p\cdot(\syd_X)_!\cdot(\syd_X)^!\cdot((\sy_X)_{\ie})^!\cdot p_{\ie !}\cdot{\inf}_{\PPdX}^!\cdot\sy_{\PPdX}&(p\ \text{satisfies \ref{Lam:f}})\\
&\leq p\cdot((\sy_X)_{\ie})^!\cdot p_{\ie !}\cdot{\inf}_{\PPdX}^!\cdot\sy_{\PPdX}&((\syd_X)_!\dv(\syd_X)^!)\\
&\leq p\cdot{\sup}_{\PPdX}\cdot(\syd_{\PPdX})^!\cdot{\inf}_{\PPdX}^!\cdot\sy_{\PPdX}&(p\ \text{satisfies \ref{Lam:g}})\\
&=p.
\end{align*}

{\bf Step 2.} As an immediate consequence of \eqref{LamQ_p}, $p$ is a right adjoint in $\QCat$. For every $\CQ$-category $X$ one has
$$\QDist(\PdX,X)\cong(\QCat)^{\co}(\PdX,\PdX)\cong(\QInf)^{\co}(\PPdX,\PX)$$
from Lemma \ref{PX_PdY_la}, with the isomorphisms given by
$$(\phi:\PdX\oto X)\mapsto(\ophi:\PdX\to\PdX)\mapsto(\phi_{\od}:\PPdX\to\PX).$$
Consequently, in order for us to establish a bijection between $\CQ$-closure operations on $\PdX$ and $(\Lam,\CQ)$-algebra structures on $X$, it suffices to prove
\begin{enumerate}[label=$\bullet$]
\item $1_{\PdX}\leq\ophi\iff\phi_{\od}$ satisfies \ref{Lam:f}, and
\item $\ophi\cdot\ophi\leq\ophi\iff\phi_{\od}$ satisfies \ref{Lam:g}
\end{enumerate}
for all $\CQ$-distributors $\phi:\PdX\oto X$.

First, $1_{\PdX}\leq\ophi\iff\phi_{\od}$ satisfies \ref{Lam:f}. Indeed,
\begin{align*}
\ora{(\syd_X)^*}=1_{\PdX}\leq\ophi&\iff\phi^{\od}\leq(\syd_X)^{*\od}=(\syd_X)_!&(\text{Lemma \ref{functor_order}\ref{phi_leq_psi}})\\
&\iff 1_{\PX}\leq\phi_{\od}\cdot(\syd_X)_!.&(\phi^{\od}\dv\phi_{\od})
\end{align*}

Second, $\ophi\cdot\ophi\leq\ophi\iff\phi_{\od}$ satisfies \ref{Lam:g}, \emph{i.e.},
$$\phi_{\od}\cdot((\sy_X)_{\ie})^!\cdot(\phi_{\od})_{\ie !}\leq\phi_{\od}\cdot{\sup}_{\PPdX}\cdot(\syd_{\PPdX})^!.$$
Since
\begin{align*}
\phi_{\od}\cdot((\sy_X)_{\ie})^!\cdot(\phi_{\od})_{\ie !}&=\phi_{\od}\cdot((\sy_X)_{\ie})^!\cdot(\phi_{\od})^{\ie !}&(\text{Lemma \ref{functor_Yoneda}\ref{functor_Yoneda:f_double}})\\
&=\phi_{\od}\cdot((\sy_X)_{\ie})^!\cdot((\phi^{\od})_{\ie})^!&(\phi^{\od}\dv\phi_{\od})\\
&=\phi_{\od}\cdot(\olphi_{\ie})^!,&(\text{Lemma \ref{dist_Yoneda}\ref{dist_Yoneda:olphi}})
\end{align*}
and since from \eqref{Fs_def} one already knows
$$\Fs_X={\sup}_{\PPdX}\cdot(\syd_{\PPdX})^!=(\syd_{\PdX})^!\cdot((\sy_{\PdX})_{\ie})^!,$$
the condition \ref{Lam:g} for $\phi_{\od}$ may be alternatively expressed as
$$\phi_{\od}\cdot(\olphi_{\ie})^!\leq\phi_{\od}\cdot(\syd_{\PdX})^!\cdot((\sy_{\PdX})_{\ie})^!.$$
Moreover, from Lemma \ref{dist_Yoneda}\ref{dist_Yoneda:phi} one has
\begin{equation} \label{olphi_ophi}
\phi^{\opl}=(\olphi^*\circ(\sy_X)_*)^{\opl}=\olphi^{\ie}\cdot(\sy_{\PdX})_{\ie},
\end{equation}
and, consequently,
\begin{align*}
&\ophi\cdot\ophi\leq\ophi\\
\iff{}&\phi^{\opl}\cdot\syd_{\PdX}\cdot\phi^{\opl}\cdot\syd_{\PdX}\leq\phi^{\opl}\cdot\syd_{\PdX}&(\text{Lemma \ref{dist_Yoneda}\ref{dist_Yoneda:olphi}})\\
\iff{}&\phi^{\opl}\cdot\syd_{\PdX}\cdot\phi^{\opl}\leq\phi^{\opl}&(\text{Lemma \ref{fyx_leq_gyx}})\\
\iff{}&\phi^{\opl}\cdot(\phi^{\opl})_{\ie}\cdot\syd_{\PdPdX}\leq\phi^{\opl}=\phi^{\opl}\cdot{\inf}_{\PdPdX}\cdot\syd_{\PdPdX}&(\syd\ \text{is natural})\\
\iff{}&\phi^{\opl}\cdot(\phi^{\opl})_{\ie}\leq\phi^{\opl}\cdot{\inf}_{\PdPdX}=\phi^{\opl}\cdot(\syd_{\PdX})^{\ie}&(\text{Lemma \ref{fyx_leq_gyx}})\\
\iff{}&(\phi\circ(\phi^{\opl})_*)^{\opl}\leq(\phi\circ(\syd_{\PdX})^*)^{\opl}\\
\iff{}&(\phi\circ(\syd_{\PdX})^*)^{\od}\leq(\phi\circ(\phi^{\opl})_*)^{\od}&(\text{Lemma \ref{functor_order}\ref{phi_leq_psi}})\\
\iff{}&(\syd_{\PdX})_!\cdot\phi^{\od}\leq(\phi^{\opl})^!\cdot\phi^{\od}=((\sy_{\PdX})_{\ie})^!\cdot\olphi^{\ie !}\cdot\phi^{\od}&(\text{Equation \eqref{olphi_ophi}})\\
\iff{}&\phi_{\od}\cdot(\olphi_{\ie})^!\leq\phi_{\od}\cdot(\syd_{\PdX})^!\cdot((\sy_{\PdX})_{\ie})^!\\
\iff{}&\phi_{\od}\ \text{satisfies \ref{Lam:g}};
\end{align*}
here the penultimate equivalence is an immediate consequence of
$$(\syd_{\PdX})_!\cdot\phi^{\od}\dv\phi_{\od}\cdot(\syd_{\PdX})^!\quad\text{and}\quad\olphi^{\ie !}\cdot\phi^{\od}\dv\phi_{\od}\cdot(\olphi_{\ie})^!.$$

{\bf Step 3.} For any $\psi:\PdY\oto Y$, $f:(X,\ophi)\to(Y,\opsi)$ is a continuous $\CQ$-functor if, and only if, $f$ as a morphism $(X,\phi_{\od})\to(Y,\psi_{\od})$ satisfies \ref{h}, \emph{i.e.},
$$f_!\cdot\phi_{\od}\leq\psi_{\od}\cdot f_{\ie !}.$$
Indeed,
\begin{align*}
&\ophi\cdot f^{\ie}\leq f^{\ie}\cdot\opsi\\
\iff{}&\phi^{\opl}\cdot\syd_{\PdX}\cdot f^{\ie}\leq f^{\ie}\cdot\psi^{\opl}\cdot\syd_{\PdY}&(\text{Lemma \ref{dist_Yoneda}\ref{dist_Yoneda:olphi}})\\
\iff{}&\phi^{\opl}\cdot(f^{\ie})_{\ie}\cdot\syd_{\PdY}\leq f^{\ie}\cdot\psi^{\opl}\cdot\syd_{\PdY}&(\syd\ \text{is natural})\\
\iff{}&\phi^{\opl}\cdot(f^{\ie})_{\ie}\leq f^{\ie}\cdot\psi^{\opl}&(\text{Lemma \ref{fyx_leq_gyx}})\\
\iff{}&\phi^{\opl}\cdot(f_{\ie})^{\ie}\leq f^{\ie}\cdot\psi^{\opl}&(\text{Lemma \ref{functor_Yoneda}\ref{functor_Yoneda:f_double}})\\
\iff{}&(\phi\circ(f_{\ie})^*)^{\opl}\leq(f^*\circ\psi)^{\opl}\\
\iff{}&(f^*\circ\psi)^{\od}\leq(\phi\circ(f_{\ie})^*)^{\od}&(\text{Lemma \ref{functor_order}\ref{phi_leq_psi}})\\
\iff{}&\psi^{\od}\cdot f_!\leq f_{\ie !}\cdot\phi^{\od}\\
\iff{}&f_!\cdot\phi_{\od}\leq\psi_{\od}\cdot f_{\ie !};&(\phi^{\od}\dv\phi_{\od}\ \text{and}\ \psi^{\od}\dv\psi_{\od})
\end{align*}
here Lemma \ref{fyx_leq_gyx} is applicable to the third equivalence because $f^{\ie}=(f^*)^{\opl}$ and $\psi^{\opl}$ are right adjoints in $\QCat$. This completes the proof.
\end{proof}

\section{The distributive law of the double copresheaf 2-monad} \label{Law_PdP}

Considering the dual of the adjunction \eqref{Pdc_dv_Pd},
\begin{equation} \label{Pc_dv_Pdc}
\sP_{\rc}^{\co\op}\dv(\sPd_{\rc})^{\co\op}:\QCat\to(\QCat)^{\co\op},
\end{equation}
one naturally constructs the \emph{double copresheaf 2-monad}
$$\bbP^{\dag}\bbP=(\PdP,\Fyd,\Fsd)$$
on $\QCat$, with the units given by
\begin{equation} \label{Fyd_def}
\Fyd_X=\syd_{\PX}\cdot\sy_X=(\sy_X)_{\ie}\cdot\syd_X:X\to\PdPX
\end{equation}
and the multiplication by
\begin{equation} \label{Fsd_def}
\Fsd_X=(\sy_{\PdPX}\cdot\syd_{\PX})^{\ie}=((\syd_{\PX})_!\cdot\sy_{\PX})^{\ie}=\ssd_{\PX}\cdot\sy_{\PdPX}^{\ie}:\PdP\PdPX\to\PdPX.
\end{equation}
Lemma \ref{fully_faithful_presheaf} and Theorem \ref{Yoneda_flat_dist_law} imply:

\begin{prop} \label{PdP_dist}
The double copresheaf 2-monad $\bbP^{\dag}\bbP$ distributes flatly over $\bbP$ by $\Lamd$ with
\[\Lamd_X=\ola{((\sy_X)_{!\ie})_*}=((\sy_X)_{!\ie})^!\cdot\sy_{\PPd\PX}=\sy_X^{!\ie !}\cdot\sy_{\PdP\PX}:\PdP\PX\to\sP\PdPX.\]
\end{prop}

It turns out that the lax $\Lambda^{\dagger}$-algebras over $\CQ$ coincide with the lax $\lam$-algebras of Theorem \ref{P_Alg}:

\begin{thm} \label{PdP_Alg}
$\LamdQAlg\cong\QCls$.
\end{thm}

\begin{proof}
{\bf Step 1.} We show that, if $(X,p)$ is a $(\Lamd,\CQ)$-algebra, then
\begin{equation} \label{LamdQ_p}
p={\inf}_{\PX}\cdot p_{\ie}\cdot(\syd_{\PX})_{\ie}.
\end{equation}
Indeed, with \eqref{Fyd_def} and \eqref{Fsd_def} one may translate the conditions \ref{f} and \ref{g} respectively as
$$\sy_X\leq p\cdot\syd_{\PX}\cdot\sy_X\quad\text{and}\quad{\sup}_{\PX}\cdot p_!\cdot\Lamd_X\cdot p_{!\ie}\leq p\cdot(\syd_{\PX})^{\ie}\cdot\sy_{\PdPX}^{\ie}.$$
To simplify the above conditions, first note that Lemma \ref{fyx_leq_gyx} implies
$$\sy_X\leq p\cdot\syd_{\PX}\cdot\sy_X\iff 1_{\PX}\leq p\cdot\syd_{\PX}.$$
Second, from Lemma \ref{functor_Yoneda}\ref{functor_Yoneda:f_double} and the naturality of $\sy$ one has
$$\Lamd_X=\sy_X^{!\ie !}\cdot\sy_{\PdP\PX}=(\sy_X^!)_{\ie !}\cdot\sy_{\PdP\PX}=\sy_{\PdPX}\cdot(\sy_X^!)_{\ie}=\sy_{\PdPX}\cdot({\sup}_{\PX})_{\ie},$$
which induces
\begin{align*}
{\sup}_{\PX}\cdot p_!\cdot\Lamd_X\cdot p_{!\ie}&={\sup}_{\PX}\cdot p_!\cdot\sy_{\PdPX}\cdot({\sup}_{\PX})_{\ie}\cdot p_{!\ie}\\
&={\sup}_{\PX}\cdot\sy_{\PX}\cdot p\cdot({\sup}_{\PX})_{\ie}\cdot p_{!\ie}&(\sy\ \text{is natural})\\
&=p\cdot({\sup}_{\PX})_{\ie}\cdot p_{!\ie}
\end{align*}
and, moreover,
\begin{align*}
&p\cdot({\sup}_{\PX})_{\ie}\cdot p_{!\ie}\leq p\cdot(\syd_{\PX})^{\ie}\cdot\sy_{\PdPX}^{\ie}\\
\iff{}&p\cdot({\sup}_{\PX})_{\ie}\cdot p_{!\ie}\cdot(\sy_{\PdPX})_{\ie}\leq p\cdot(\syd_{\PX})^{\ie}&(\sy_{\PdPX}^{\ie}\dv(\sy_{\PdPX})_{\ie})\\
\iff{}&p\cdot({\sup}_{\PX})_{\ie}\cdot(\sy_{\PX})_{\ie}\cdot p_{\ie}\leq p\cdot(\syd_{\PX})^{\ie}&(\sy\ \text{is natural})\\
\iff{}&p\cdot p_{\ie}\leq p\cdot(\syd_{\PX})^{\ie}.
\end{align*}
Therefore, $(X,p)$ is a $(\Lamd,\CQ)$-algebra if, and only if,
\begin{enumerate}[label=(\alph*),start=6]
\item \label{Lamd:f}
    $1_{\PX}\leq p\cdot\syd_{\PX}$ and
\item \label{Lamd:g}
    $p\cdot p_{\ie}\leq p\cdot(\syd_{\PX})^{\ie}$.
\end{enumerate}
It follows that
\begin{align*}
p&={\inf}_{\PX}\cdot\syd_{\PX}\cdot p\\
&={\inf}_{\PX}\cdot p_{\ie}\cdot\syd_{\PdPX}&(\syd\ \text{is natural})\\
&\leq{\inf}_{\PX}\cdot p_{\ie}\cdot(\syd_{\PX})_{\ie}&(\bbP^{\dag}\ \text{is oplax idempotent})\\
&\leq p\cdot\syd_{\PX}\cdot{\inf}_{\PX}\cdot p_{\ie}\cdot(\syd_{\PX})_{\ie}&(p\ \text{satisfies \ref{Lamd:f}})\\
&\leq p\cdot p_{\ie}\cdot(\syd_{\PX})_{\ie}&(\syd_{\PX}\dv{\inf}_{\PX})\\
&\leq p\cdot(\syd_{\PX})^{\ie}\cdot(\syd_{\PX})_{\ie}&(p\ \text{satisfies \ref{Lamd:g}})\\
&=p.&(\syd_{\PX}\ \text{is fully faithful})
\end{align*}

{\bf Step 2.} As an immediate consequence of \eqref{LamdQ_p}, $p$ is a right adjoint in $\QCat$. For every $\CQ$-category $X$, as one already has
$$\QDist(X,\PX)\cong\QCat(\PX,\PX)\cong\QInf(\PdPX,\PX)$$
from Lemma \ref{PX_PdY_la}, with the isomorphisms given by
$$(\phi:X\oto\PX)\mapsto(\olphi:\PX\to\PX)\mapsto(\dphi:\PdPX\to\PX),$$
in order for us to establish a bijection between $\CQ$-closure operations on $\PX$ and $(\Lamd,\CQ)$-algebra structures on $X$, it suffices to prove
\begin{enumerate}[label=$\bullet$]
\item $1_{\PX}\leq\olphi\iff\dphi$ satisfies \ref{Lamd:f}, and
\item $\olphi\cdot\olphi\leq\olphi\iff\dphi$ satisfies \ref{Lamd:g}
\end{enumerate}
for all $\CQ$-distributors $\phi:X\oto\PX$.

First, the equivalence $(1_{\PX}\leq\olphi\iff\dphi$ satisfies \ref{Lamd:f}) holds trivially since $\olphi=\dphi\cdot\syd_{\PX}$.

Second, one has ($\olphi\cdot\olphi\leq\olphi\iff\dphi$ satisfies \ref{Lamd:g}). Indeed,
\begin{align*}
\olphi\cdot\olphi\leq\olphi&\iff\dphi\cdot\syd_{\PX}\cdot\dphi\cdot\syd_{\PX}\leq\dphi\cdot\syd_{\PX}&(\olphi=\dphi\cdot\syd_{\PX})\\
&\iff\dphi\cdot\syd_{\PX}\cdot\dphi\leq\dphi&(\text{Lemma \ref{fyx_leq_gyx}})\\
&\iff\dphi\cdot(\dphi)_{\ie}\cdot\syd_{\PdPX}\leq\dphi=\dphi\cdot{\inf}_{\PdPX}\cdot\syd_{\PdPX}&(\syd\ \text{is natural})\\
&\iff\dphi\cdot(\dphi)_{\ie}\leq\dphi\cdot{\inf}_{\PdPX}=\dphi\cdot(\syd_{\PX})^{\ie}&(\text{Lemma \ref{fyx_leq_gyx}})\\
&\iff\dphi\ \text{satisfies \ref{Lamd:g}}.
\end{align*}

{\bf Step 3.} For any $\psi:Y\oto\PY$, $f:(X,\olphi)\to(Y,\olpsi)$ is a continuous $\CQ$-functor if, and only if, $f:(X,\dphi)\to(Y,\dpsi)$ satisfies \ref{h}. Indeed,
\begin{align*}
f_!\cdot\olphi\leq\olpsi\cdot f_!&\iff f_!\cdot\dphi\cdot\syd_{\PX}\leq\dpsi\cdot\syd_{\PY}\cdot f_!\\
&\iff f_!\cdot\dphi\cdot\syd_{\PX}\leq\dpsi\cdot f_{!\ie}\cdot\syd_{\PX}&(\syd\ \text{is natural})\\
&\iff f_!\cdot\dphi\leq\dpsi\cdot f_{!\ie},&(\text{Lemma \ref{fyx_leq_gyx}})
\end{align*}
which completes the proof.
\end{proof}

\section{Distributive laws of $\bbT$ over $\bbP$ versus lax extensions of $\bbT$ to $\QDist$} \label{Dist_Law_vs_Extension}

In this section, for an arbitrary 2-monad $\bbT$ on $\QCat$, we outline the bijective correspondence between distributive laws\footnote{We remind the reader that, as stated in Section \ref{Lax_Dist_Laws}, in this paper we use ``distributive law'' to mean ``\emph{lax} distributive law'', which is especially relevant when reading Proposition \ref{lax_ext_monad} and Corollary \ref{flat_lax_ext_monad} below.} of $\bbT$ over $\bbP$ and so-called lax extensions of $\bbT$ to $\QDist$. The techniques adopted here generalize their discrete counterparts as given in \cite{Tholen2016}.

Given a 2-functor $T:\QCat\to\QCat$, a \emph{lax extension of} $T$ to $\QDist$ is a lax functor
$$\hT:\QDist\to\QDist$$
that coincides with $T$ on objects and satisfies the extension condition (3) below. Explicitly, $\hT$ is given by a family
\begin{equation} \label{hTphi}
(\hT\phi:TX\oto TY)_{\phi\in\QDist(X,Y),\ X,Y\in\ob(\QCat)}
\end{equation}
of $\CQ$-distributors such that
\begin{enumerate}[label=(\arabic*)]
\item \label{one}
    $\phi\preceq\phi'{}\Lra{}\hT\phi\preceq\hT\phi'$,
\item \label{two}
    $\hT\psi\circ\hT\phi\preceq\hT(\psi\circ\phi)$,
\item \label{three}
    $(Tf)_*\preceq\hT(f_*)$,\quad $(Tf)^*\preceq\hT(f^*)$,
\end{enumerate}
for all $\CQ$-distributors $\phi,\phi':X\oto Y$, $\psi:Y\oto Z$ and $\CQ$-functors $f:X\to Y$.

It is useful to present the following equivalent conditions of \ref{three}, which can be proved analogously to their discrete versions in \cite{Tholen2016}, by straightforward calculation:

\begin{lem} \label{hT_graph}
Given a family \eqref{hTphi} of $\CQ$-distributors satisfying {\rm\ref{one}} and {\rm\ref{two}}, the following conditions are equivalent when quantified over the variables occurring in them ($f:X\to Y$, $\phi:Z\oto Y$, $\psi:Y\oto Z$):
\begin{enumerate}[label={\rm (\roman*)}]
\item \label{hT_graph:cograph}
    $1_{TX}^*\preceq\hT(1_X^*)$,\quad $\hT(f^*\circ\phi)=(Tf)^*\circ\hT\phi$.
\item \label{hT_graph:graph}
    $1_{TX}^*\preceq\hT(1_X^*)$,\quad $\hT(\psi\circ f_*)=\hT\psi\circ(Tf)_*$.
\item \label{hT_graph:three}
    $(Tf)_*\preceq\hT(f_*)$,\quad $(Tf)^*\preceq\hT(f^*)$\quad (\emph{i.e.}, $\hT$ satisfies {\rm\ref{three}}).
\end{enumerate}
\end{lem}

%
%

\begin{prop} \label{lax_ext_functor}
Lax extensions of a 2-functor $T:\QCat\to\QCat$ to $\QDist$ correspond bijectively to lax natural transformations $T\sP\to\sP T$ satisfying the lax $\bbP$-unit law and the lax $\bbP$-multiplication law.
\end{prop}

\begin{proof}
{\bf Step 1.} For each $\lam:T\sP\to\sP T$ satisfying \ref{a}, \ref{b} and \ref{c}, $\Phi(\lam):=\hT=(\hT\phi)_{\phi}$ with $\ola{\hT\phi}:=\lam_X\cdot T\olphi$ is a lax extension of $T$ to $\QDist$.
$$\begin{array}{cccc}
\Phi(\lam)=\hT: & \QDist(X,Y) & \to & \QDist(TX,TY)\\
 & (\olphi:Y\to\PX) & \mapsto & \bfig\Vtriangle/->`->`<-/<300,300>[TY`\sP TX`T\PX;\ola{\hT\phi}`T\olphi`\lam_X]\efig
\end{array}$$
Indeed, \ref{one} follows immediately from the 2-functoriality of $T$. For \ref{two}, just note that
\begin{align*}
\ola{\hT\psi\circ\hT\phi}&=\sy_{TX}^!\cdot(\ola{\hT\phi})_!\cdot\ola{\hT\psi}&(\text{Lemma \ref{transpose}\ref{transpose:l}})\\
&=\sy_{TX}^!\cdot(\lam_X)_!\cdot(T\olphi)_!\cdot\lam_Y\cdot T\olpsi\\
&\leq\sy_{TX}^!\cdot(\lam_X)_!\cdot\lam_{\PX}\cdot T(\olphi_!)\cdot T\olpsi&(\lam\ \text{satisfies \ref{a}})\\
&\leq\lam_X\cdot T\sy_X^!\cdot T(\olphi_!)\cdot T\olpsi&(\lam\ \text{satisfies \ref{c}})\\
&=\lam_X\cdot T(\ola{\psi\circ\phi})&(\text{Lemma \ref{transpose}\ref{transpose:l}})\\
&=\ola{\hT(\psi\circ\phi)}.
\end{align*}
For \ref{three}, it suffices to check Lemma \ref{hT_graph}\ref{hT_graph:cograph}. Since $\lam$ satisfies \ref{b}, it follows easily that
$$\ola{1_{TX}^*}=\sy_{TX}\leq\lam_X\cdot T\sy_X=\lam_X\cdot T\ola{1_X^*}=\ola{\hT(1_X^*)}.$$
For the second identity, Lemma \ref{transpose}\ref{transpose:l} implies
$$\ola{\hT(f^*\circ\phi)}=\lam_X\cdot T(\ola{f^*\circ\phi})=\lam_X\cdot T\olphi\cdot Tf=\ola{\hT\phi}\cdot Tf=\ola{(Tf)^*\circ\hT\phi}.$$

{\bf Step 2.} For every lax extension $\hT$ of $T$, $\Psi(\hT):=\lam=(\lam_X)_X$ with
$$\lam_X:=\ola{\hT\ep_X}=\ola{\hT(\sy_X)_*}:T\PX\to\sP TX$$
is a lax natural transformation satisfying the $\bbP$-unit law and the $\bbP$-multiplication law.

\ref{a} $(Tf)_!\cdot\lam_X\leq\lam_Y\cdot T(f_!)$ for all $\CQ$-functors $f:X\to Y$. Indeed,
\begin{align*}
(Tf)_!\cdot\lam_X&=\ola{\hT(\sy_X)_*\circ(Tf)^*}&(\text{Lemma \ref{transpose}\ref{transpose:l}})\\
&\leq\ola{\hT(\sy_X)_*\circ(Tf)^*\circ\hT(1_Y^*)}&(\text{Lemma \ref{hT_graph}\ref{hT_graph:cograph}})\\
&=\ola{\hT(\sy_X)_*\circ\hT(f^*)}&(\text{Lemma \ref{hT_graph}\ref{hT_graph:cograph}})\\
&\leq\ola{\hT((\sy_X)_*\circ f^*)}&(\hT\ \text{satisfies \ref{two}})\\
&=\ola{\hT((f_!)^*\circ(\sy_Y)_*)}&(\text{Lemma \ref{functor_Yoneda}\ref{functor_Yoneda:yf}})\\
&=\ola{(Tf_!)^*\circ\hT(\sy_Y)_*}&(\text{Lemma \ref{hT_graph}\ref{hT_graph:cograph}})\\
&=\lam_Y\cdot T(f_!).&(\text{Lemma \ref{transpose}\ref{transpose:l}})
\end{align*}

\ref{b} $\sy_{TX}\leq\lam_X\cdot T\sy_X$. Indeed,
\begin{align*}
\sy_{TX}=\ola{1_{TX}^*}&\leq\ola{\hT(1_X^*)}&(\text{Lemma \ref{hT_graph}\ref{hT_graph:cograph}})\\
&=\ola{\hT(\sy_X^*\circ(\sy_X)_*)}&(\sy_X\ \text{is fully faithful})\\
&=\ola{(T\sy_X)^*\circ\hT(\sy_X)_*}&(\text{Lemma \ref{hT_graph}\ref{hT_graph:cograph}})\\
&=\lam_X\cdot T\sy_X.&(\text{Lemma \ref{transpose}\ref{transpose:l}})
\end{align*}

\ref{c} $\sfs_{TX}\cdot(\lam_X)_!\cdot\lam_{\PX}\leq\lam_X\cdot T\sfs_X$. Indeed,
\begin{align*}
\sfs_{TX}\cdot(\lam_X)_!\cdot\lam_{\PX}&=\ola{\hT(\sy_{\PX})_*\circ\hT(\sy_X)_*}&(\text{Lemma \ref{transpose}\ref{transpose:l}})\\
&\leq\ola{\hT((\sy_{\PX})_*\circ(\sy_X)_*)}&(\hT\ \text{satisfies \ref{two}})\\
&=\ola{\hT(((\sy_X)_!)_*\circ(\sy_X)_*)}&(\sy\ \text{is natural})\\
&=\ola{\hT((\sy_X^!)^*\circ(\sy_X)_*)}&((\sy_X)_!\dv\sy_X^!)\\
&=\ola{(T\sy_X^!)^*\circ\hT(\sy_X)_*}&(\text{Lemma \ref{hT_graph}\ref{hT_graph:cograph}})\\
&=\lam_X\cdot T\sfs_X.&(\text{Lemma \ref{transpose}\ref{transpose:l}})
\end{align*}

{\bf Step 3.} $\Phi$ and $\Psi$ are inverse to each other. For each $\lam:T\sP\to\sP T$, $\Psi\Phi(\lam)=\lam$ since
$$(\Psi\Phi(\lam))_X=\ola{\Phi(\lam)(\sy_X)_*}=\lam_X\cdot T\ola{(\sy_X)_*}=\lam_X\cdot T1_{\PX}=\lam_X.$$
Conversely, for every lax extension $\hT$, one has
$$\ola{(\Phi\Psi(\hT))\phi}=\ola{\hT(\sy_X)_*}\cdot T\olphi=\ola{(T\olphi)^*\circ\hT(\sy_X)_*}=\ola{\hT(\olphi^*\circ(\sy_X)_*)}=\ola{\hT\phi},$$
where the last three equalities follow respectively from Lemmas \ref{transpose}\ref{transpose:l}, \ref{hT_graph}\ref{hT_graph:cograph} and \ref{dist_Yoneda}\ref{dist_Yoneda:phi}.
\end{proof}

For a 2-monad $\bbT=(T,m,e)$ on $\QCat$, a lax extension $\hT$ of $T$ to $\QDist$ becomes a \emph{lax extension of the 2-monad} $\bbT$ if it further satisfies
\begin{enumerate}[label=(\arabic*),start=4]
\item \label{four} $\phi\circ e_X^*\preceq e_Y^*\circ\hT\phi$,
\item \label{five} $\hT\hT\phi\circ m_X^*\preceq m_Y^*\circ\hT\phi$
\end{enumerate}
for all $\CQ$-distributors $\phi:X\oto Y$. By adjunction, \ref{four} and \ref{five} may be equivalently expressed as
\begin{enumerate}[label=(\arabic*'),start=4]
\item \label{fourp} $(e_Y)_*\circ\phi\preceq\hT\phi\circ(e_X)_*$,
\item \label{fivep} $(m_Y)_*\circ\hT\hT\phi\preceq\hT\phi\circ(m_X)_*$.
\end{enumerate}

\begin{prop} \label{lax_ext_monad}
Lax extensions of a 2-monad $\bbT=(T,m,e)$ on $\QCat$ to $\QDist$ correspond bijectively to distributive laws of $\bbT$ over $\bbP$.
\end{prop}

\begin{proof}
With Proposition \ref{lax_ext_functor} at hand, it suffices to prove
\begin{enumerate}[label=$\bullet$]
\item $\hT$ satisfies \ref{four} $\iff$ $\lam$ satisfies \ref{d}, and
\item $\hT$ satisfies \ref{five} $\iff$ $\lam$ satisfies \ref{e}
\end{enumerate}
for every lax extension $\hT$ of the 2-functor $T$ and $\lam=\Psi(\hT)$ with $\lam_X=\ola{\hT(\sy_X)_*}:T\PX\to\sP TX$.

First, ($\hT$ satisfies \ref{four} $\iff$ $\lam$ satisfies \ref{d}). Since Lemma \ref{transpose}\ref{transpose:l} and the naturality of $e$ imply
$$(e_X)_!\cdot\olphi=\ola{\phi\circ e_X^*}\quad\text{and}\quad\lam_X\cdot e_{\PX}\cdot\olphi=\lam_X\cdot T\olphi\cdot e_Y=\ola{\hT\phi}\cdot e_Y=\ola{e_Y^*\circ\hT\phi}$$
for all $\phi:X\oto Y$, it follows that
\begin{align*}
(e_X)_!\leq\lam_X\cdot e_{\PX}&\iff\forall\phi:X\oto Y:\ (e_X)_!\cdot\olphi\leq\lam_X\cdot e_{\PX}\cdot\olphi\\
&\iff\forall\phi:X\oto Y:\ \phi\circ e_X^*\preceq e_Y^*\circ\hT\phi.
\end{align*}

Second, ($\hT$ satisfies \ref{five} $\iff$ $\lam$ satisfies \ref{e}). Similarly as above, one has
$$(m_X)_!\cdot\lam_{TX}\cdot T\lam_X\cdot TT\olphi=(m_X)_!\cdot\lam_{TX}\cdot T\ola{\hT\phi}=(m_X)_!\cdot\ola{\hT\hT\phi}=\ola{\hT\hT\phi\circ m_X^*}$$
and
$$\lam_X\cdot m_{\PX}\cdot TT\olphi=\lam_X\cdot T\olphi\cdot m_Y=\ola{\hT\phi}\cdot m_Y=\ola{m_Y^*\circ\hT\phi}$$
by Lemma \ref{transpose}\ref{transpose:l} and the naturality of $m$. Consequently,
\begin{align*}
&(m_X)_!\cdot\lam_{TX}\cdot T\lam_X\leq\lam_X\cdot m_{\PX}\\
\iff{}&\forall\phi:X\oto Y:\ (m_X)_!\cdot\lam_{TX}\cdot T\lam_X\cdot TT\olphi\leq\lam_X\cdot m_{\PX}\cdot TT\olphi\\
\iff{}&\forall\phi:X\oto Y:\ \hT\hT\phi\circ m_X^*\preceq m_Y^*\circ\hT\phi.
\end{align*}
\end{proof}

A lax extension $\hT$ of a 2-functor $T$ on $\QCat$ is \emph{flat} (more commonly known as \emph{normal}) if
$$\hT 1_X^*=1_{TX}^*$$
for all $\CQ$-categories $X$. One says that a lax extension $\hT$ of a 2-monad $\bbT=(T,m,e)$ on $\QCat$ is \emph{flat} if $\hT$, as a lax extension of the 2-functor $T$, is flat.

If $\hT$ is related to a distributive law $\lam$ by the correspondence of Proposition \ref{lax_ext_monad}, then with Lemma \ref{dist_Yoneda}\ref{dist_Yoneda:y} one sees immediately
$$\ola{\hT 1_X^*}=\lam_X\cdot T\ola{1_X^*}=\lam_X\cdot T\sy_X\quad\text{and}\quad\ola{1_{TX}^*}=\sy_{TX}.$$
Thus we have proved:

\begin{cor} \label{flat_lax_ext_monad}
Flat lax extensions of a $\bbT$ to $\QDist$ correspond bijectively to flat distributive laws of $\bbT$ over $\bbP$.
\end{cor}

When formulated equivalently in terms of lax monad extensions, Theorem \ref{Yoneda_flat_dist_law} gives \cite[Theorem 4.4]{Lai2016a}, since $\lam_X=\ola{(T\sy_X)_*}$ in Theorem \ref{Yoneda_flat_dist_law} corresponds to the lax extension $\hT$ with $\ola{\hT\phi}=\lam_X\cdot T\olphi=\ola{(T\olphi)^*\circ(T\sy_X)_*}$. In summary, we obtain the following theorem.

\begin{thm}\label{Yoneda_flat_dist_extension}
Let $T$ be a 2-functor on $\QCat$.
Then
$$\hT\phi=(T\olphi)^*\circ(T\sy_X)_*:TX\oto TY$$
defines a lax extension of $T$ to $\QDist$, and the following statements are equivalent:
\begin{enumerate}[label={\rm (\roman*)}]
\item $\hT$ is flat;
\item there exists some flat lax extension of $T$ to $\QDist$;
\item $T$ maps fully faithful $\CQ$-functors to fully faithful $\CQ$-functors;
\item $T\sy_X$ is fully faithful for all $\CQ$-categories $X$;
\item $\hT$ is the only flat lax extension of $T$ to $\QDist$.
\end{enumerate}
Moreover, if $T$ belongs to a 2-monad $\bbT$ on $\QCat$ and satisfies the above equivalent conditions, then $\hT$ is a flat lax extension of the 2-monad $\bbT$, and it is the only one.
\end{thm}

\begin{proof}
By checking the proofs of Proposition \ref{flat_Tyx} and Theorem \ref{Yoneda_flat_dist_law}, it is not difficult to extract from them the corresponding conclusions that are valid for all $\lam:T\sP\to\sP T$ satisfying \ref{a}, \ref{b} and \ref{c} which, by Proposition \ref{lax_ext_functor}, may be transferred to lax extensions of 2-functors; that is, $\hT$ always defines a lax extension of the 2-functor $T$, and (iv)${}\Lra{}$(v) holds. Now the only non-trivial part of the proof of (i)${}\Lra{}$(ii)${}\Lra{}$(iii)${}\Lra{}$(iv)${}\Lra{}$(v)${}\Lra{}$(i) not yet covered is the implication (ii)${}\Lra{}$(iii). But if a $\CQ$-functor $f:X\to Y$ satisfies $f^*\circ f_*=1_X^*$, the application of any flat lax extension $\tilde{T}$ of $T$ to this equality gives $$(Tf)^*\circ(Tf)_*=(Tf)^*\circ 1_{TX}^*\circ(Tf)_*=(Tf)^*\circ \tilde{T}1_X^*\circ (Tf)_*=\tilde{T}(f^*\circ f_*)=1_{TX}^*,$$
by Lemma \ref{hT_graph}.
\end{proof}

\begin{rem}
\begin{enumerate}[label=(\Roman*)]
\item \label{phi_u_v}
    Let us mention that, more generally than in the above proof, whenever $\phi=v^*\circ u_*:X\oto Y$ with $\CQ$-functors $u:X\to Z, v:Y\to Z$, one has
    $$\tilde{T}\phi=(Tv)^*\circ(Tu)_*,$$
    for every flat lax extension $\tilde{T}$ of $\bbT$ --- which must then coincide with $\hT$ if there is such $\tilde{T}$.
\item As the anonymous referee observed, as a consequence of \ref{phi_u_v}  one has the following presentation of the uniquely determined flat lax extension $\hT$ when $T$ preserves the full fidelity of $\CQ$-functors: simply take for $u$ and $v$ above the fully faithful injections of respectively $X$ and $Y$ into the \emph{collage} (also \emph{cograph}) of $\phi$: its objects are given by the disjoint union of the object sets of $X$ and $Y$, and the hom arrows from objects in $X$ to objects in $Y$ are given by $\phi$, while in the opposite direction they are always bottom element arrows.
\item For the sake of completeness let us also mention the following obvious extension of the language used in the context of lax distributive laws: a \emph{strict} extension of a 2-monad $\bbT=(T,m,e)$ on $\QCat$ is a 2-functor $\hT:\QDist\to\QDist$ that coincides with $T$ on objects and satisfies
    \begin{enumerate}[label={\rm (\arabic*$\ast$)},start=3]
    \item $\hT(f^*\circ\phi)=(Tf)^*\circ\hT\phi$,
    \item $\phi\circ e_X^*=e_Y^*\circ\hT\phi$,
    \item $\hT\hT\phi\circ m_X^*=m_Y^*\circ\hT\phi$,
    \end{enumerate}
    for all $f$ and $\phi$. In other words, a lax extension $\hT$ of $\bbT$ is strict if all the inequalities in \ref{two}, Lemma \ref{hT_graph}\ref{hT_graph:cograph}, \ref{four} and \ref{five} are equalities. From the above proofs one immediately sees that strict extensions of $\bbT$ to $\QDist$ correspond bijectively to strict distributive laws of $\bbT$ over $\bbP$.
\end{enumerate}
\end{rem}

With a given lax extension $\hT$ to $\QDist$ of the 2-monad $\bbT$ of $\QCat$ we can now define:

\begin{defn}
A \emph{$\TQ$-category} $(X,\al)$ consists of a $\CQ$-category $X$ and a $\CQ$-distributor $\al:X\oto TX$ satisfying the lax unit and lax multiplication laws
$$1_X^*\preceq e_X^*\circ\al\quad\text{and}\quad\hT\al\circ\al\preceq m_X^*\circ\al.$$
A \emph{$\TQ$-functor} $f:(X,\al)\to(Y,\be)$ is a $\CQ$-functor $f:X\to Y$ with
$$\al\circ f^*\preceq(Tf)^*\circ\be.$$
\end{defn}

$\TQ$-categories and $\TQ$-functors constitute a 2-category $\TQCat$, which is more precisely recorded as $\TTQCat$. It is not surprising that this category just disguises $\lamQAlg$ (and vice versa), for $\lam$ corresponding to $\hT$. The discrete counterpart of this fact already appeared in \cite{Tholen2016}.

\begin{cor} \label{lamQAlg_iso_TTQCat}
If $\lam$ and $\hT$ are related by the correspondence of Proposition {\rm\ref{lax_ext_monad}}, then
$$\lamQAlg\cong\TTQCat.$$
\end{cor}

\begin{proof}
For any $\CQ$-category $X$, as one already has
$$\QDist(X,TX)\cong\QCat(TX,\PX),$$
with the isomorphism given by
$$(\al:X\oto TX)\mapsto(\olal:TX\to\PX),$$
in order for us to establish a bijection between $\TQ$-category structures on $X$ and $(\lam,\CQ)$-algebra structures on $X$, it suffices to prove
\begin{enumerate}[label=$\bullet$]
\item $1_X^*\preceq e_X^*\circ\al\iff\sy_X\leq\olal\cdot e_X$, and
\item $\hT\al\circ\al\preceq m_X^*\circ\al\iff\sy_X^!\cdot\olal_!\cdot\lam_X\cdot T\olal\leq\olal\cdot m_X,$
\end{enumerate}
for all $\CQ$-distributors $\al:X\oto TX$. Indeed, the first equivalence is easy since $\ola{1_X^*}=\sy_X$ and $\ola{e_X^*\circ\al}=\olal\cdot e_X$ by Lemma \ref{transpose}\ref{transpose:l}. For the second equivalence, just note that $\ola{m_X^*\circ\al}=\olal\cdot m_X$ and
\begin{align*}
\ola{\hT\al\circ\al}&=\ola{(\hT(\olal^*\circ(\sy_X)_*)\circ\al}&(\text{Lemma \ref{dist_Yoneda}\ref{dist_Yoneda:phi}})\\
&=\ola{(T\olal)^*\circ\hT(\sy_X)_*\circ\al}&(\text{Lemma \ref{hT_graph}\ref{hT_graph:cograph}})\\
&=\ola{\hT(\sy_X)_*\circ\al}\cdot T\olal&(\text{Lemma \ref{transpose}\ref{transpose:l}})\\
&=\sy_X^!\cdot\olal_!\cdot\lam_X\cdot T\olal,&(\text{Lemma \ref{transpose}\ref{transpose:l} and}\ \lam_X=\ola{\hT(\sy_X)_*})
\end{align*}

Finally, a $\CQ$-functor $f:X\to Y$ is a $\TQ$-functor $f:(X,\al)\to(Y,\be)$ if, and only if, $f:(X,\olal)\to(Y,\olbe)$ is a lax $\lam$-homomorphism since
$$\al\circ f^*\preceq(Tf)^*\circ\be\iff f_!\cdot\olal=\ola{\al\circ f^*}\leq\ola{(Tf)^*\circ\be}=\olbe\cdot Tf$$
by Lemma \ref{transpose}\ref{transpose:l}.
\end{proof}

\begin{exmp}
\begin{enumerate}[label={\rm (\arabic*)}]
\item For the identity 2-monad $\bbI$ on $\QCat$, the identity 2-functor on $\QDist$ is a strict extension of $\bbI$, and it is easy to see that $(\bbI,\CQ)\text{-}\Cat\cong\Mon(\QDist)$.
\item The flat distributive law $\lam$ of $\bbP$ over itself described in Proposition \ref{P_dist} corresponds to the flat lax extension $\hP$ of $\bbP$ with
    $$\hP\phi:=\phi^{\od *}:\PX\oto\PY$$
    for $\phi:X\oto Y$. From Theorem \ref{P_Alg} one obtains $(\bbP,\CQ)\text{-}\Cat\cong\QCls$.
\item The strict distributive law $\lamd$ of $\bbP^{\dag}$ over $\bbP$ given in Proposition \ref{Pd_dist} determines the strict extension $\check{\sP}^{\dag}$ of $\bbP^{\dag}$ with
    $$\check{\sP}^{\dag}\phi:=(\phi^{\opl})_*:\PdX\oto\PdY.$$
    Theorem \ref{Pd_Alg} shows that $(\bbP^{\dag},\CQ)\text{-}\Cat\cong\Mon(\QDist)$.
\item Proposition \ref{PPd_dist} gives the flat distributive law $\Lam$ of $\bbP\bbP^{\dag}$ over $\bbP$ that corresponds to the flat lax extension $\widehat{\PPd}$ of $\bbP\bbP^{\dag}$ with
    $$\widehat{\PPd}\phi:=\hP\check{\sP}^{\dag}\phi=((\phi^{\opl})_*)^{\od *}=\phi^{\opl ! *}:\PPdX\oto\PPdY.$$
    From Theorem \ref{PPd_Alg} one has $(\bbP\bbP^{\dag},\CQ)\text{-}\Cat\cong\QInt$.
\item The flat distributive law $\Lamd$ of $\bbP^{\dag}\bbP$ over $\bbP$ (see Proposition \ref{PdP_dist}) is related to the flat lax extension $\widehat{\PdP}$ of $\bbP^{\dag}\bbP$ with
    $$\widehat{\PdP}\phi:=\check{\sP}^{\dag}\hP\phi=(\phi^{\od *\opl})_*=(\phi^{\od\ie})_*:\PdPX\oto\PdPY.$$
    Theorem \ref{PdP_Alg} shows that $(\bbP^{\dag}\bbP,\CQ)\text{-}\Cat\cong\QCls$.
\end{enumerate}
\end{exmp}


\end{document}